\documentclass[a4paper, 12pt]{amsart}
\usepackage{mathrsfs}
\usepackage{amsmath}
\usepackage{amssymb}
\usepackage{verbatim}
\usepackage{xcolor}
\usepackage{CJK}
\usepackage{graphicx}
\usepackage[all]{xy}
\usepackage{appendix}
\DeclareGraphicsExtensions{.eps,.ps,.jpg,.bmp}

\theoremstyle{plain}

\newtheorem{theorem}{Theorem}[section]
\newtheorem{proposition}{Proposition}[section]

\newtheorem{lemma}{Lemma}[section]
\newtheorem{conjecture}{Conjecture}[section]
\newtheorem{problem}{Problem}[section]

\theoremstyle{definition}
\newtheorem{definition}{Definition}[section]

\theoremstyle{remark}
\newtheorem{remark}{Remark}[section]
\newtheorem*{acknowledgment}{Acknowledgment}
\title{An eigenvalue estimate for the $\bar{\partial}$-Laplacian associated to a nef line bundle}
\author{Jingcao Wu}

\begin{document}
\pagestyle{plain}
\begin{abstract}
We study the $\bar{\partial}$-Laplacian on forms taking values in $L^{k}$, a high power of a nef line bundle on a compact complex manifold, and give an estimate of the number of the eigenforms whose corresponding eigenvalues smaller than or equal to $\lambda$. In particular, the $\lambda=0$ case gives an asymptotic estimate for the order of the corresponding cohomology groups. It helps to generalize the Grauert--Riemenschneider conjecture. At last, we discuss the $\lambda=0$ case on a pseudo-effective line bundle.
\end{abstract}
\maketitle

\section{Introduction}
Let $X$ be a compact complex manifold of dimension $n$, and let $L$ be a holomorphic line bundle on $X$. Fix Hermitian metrics on $X$ and $L$ respectively, the classic geometry theory allows us to define the adjoint operator $\bar{\partial}^{\ast}$ of the $\bar{\partial}$-operator acting on $L$-valued forms as well as the corresponding Laplacian operator
\begin{equation*}
    \Delta=\bar{\partial}\bar{\partial}^{\ast}+\bar{\partial}^{\ast}\bar{\partial}.
\end{equation*}

The theory of the elliptic partial differential operator applies on this Laplacian operator, and it has a good counterpart in geometry due to Hodge's theorem. More specifically, the classic Hodge's theorem says that any class $[\alpha]$ in the Dolbeault cohomology group $H^{p,q}(X,L)$ owns a unique harmonic representative $\tilde{\alpha}$, i.e. $\tilde{\alpha}\in[\alpha]$ and $\Delta\tilde{\alpha}=0$. In other word, if we denote the space of harmonic $L$-valued $(p,q)$-forms by $\mathcal{H}^{p,q}(X,L)$, we have
\[
\mathcal{H}^{p,q}(X,L)\simeq H^{p,q}(X,L).
\]
It is worth to mention that when $L$ is ample, it leads to excellent results in geometry, such as the hard Lefschetz theorem.

In this paper, $L$ is taken as a nef line bundle, and we are interested in the quantity of eigenforms of the Laplacian on $L$. The work in this aspect dates back to \cite{Ber02}. Indeed, if we denote the linear space of the $L$-valued $(n,q)$-eigenforms of $\Delta$, with corresponding eigenvalues smaller than or equal to $\lambda$ by
\[
    \mathcal{H}^{n,q}_{\leqslant\lambda}(X,L),
\]
it is given in \cite{Ber02} an asymptotic estimate for
\[
    h^{n,q}_{\leqslant\lambda}(L^{k}\otimes E)=\dim \mathcal{H}^{n,q}_{\leqslant\lambda}(X,L^{k}\otimes E)
\]
in the case that $L$ is a semi-positive line bundle and $E$ is a line bundle on $X$. Also it is shown in \cite{Ber02} through an example (Proposition 4.2) that this estimate is sharp.

In this paper, we will prove a similar result when $L$ is nef. We first introduce a canonical way to define the Laplacian on a nef line bundle in the text as well as the eigenform space $\mathcal{H}^{p,q}_{\leqslant\lambda}$ (Definition \ref{d21}, Sect.2). Let
\[
h^{n,q}_{\leqslant\lambda}(L^{k}\otimes E)=\dim\mathcal{H}^{n,q}_{\leqslant\lambda}(X,L^{k}\otimes E).
\]

Then we define the so-called modified ideal sheaf (Definition \ref{d22}, Sect.2). We give a brief introduction here for readers' convenience. Notice that when $L$ is nef, for any $\varepsilon>0$, there exits a smooth metric $h_{\varepsilon}$ on $L$ such that $i\Theta_{L,h_{\varepsilon}}\geqslant-\varepsilon\omega$. Here $\omega$ is a Hermitian metric on $X$. So there is an $L^{1}$-limit $h_{0}$ of $h_{\varepsilon}$ (after passing to a subsequence) with $i\Theta_{L,h_{0}}\geqslant0$. In the rest part, $h_{0}$ will always refer to such a metric.

Now we furthermore assume that $h_{0}$ has analytic singularities. In this case the associated multiplier ideal sheaf can be computed as follows in \cite{Dem12}:
\[
\mathscr{I}(h_{0})=\mu_{\ast}\mathcal{O}_{\tilde{X}}(\sum(\rho_{j}-\lfloor a\lambda_{j}\rfloor)D_{j}).
\]
Here $\mu:\tilde{X}\rightarrow X$ is a log-resolution, and $\rho_{j},a,\lambda_{j}$ are involved real numbers. $\lfloor a\lambda_{j}\rfloor$ means the round down. The precise meaning of these notations will be clarified in the text. The modified ideal sheaf is then defined as
\[
\mathfrak{I}(h_{0}):=\mu_{\ast}\mathcal{O}_{\tilde{X}}(\sum(-\lceil a\lambda_{j}\rceil)D_{j}).
\]
$\lceil a\lambda_{j}\rceil$ refers to the round up. We will also denote it by $\mathfrak{I}(L)$. It is not hard to see that
\[
\mathfrak{I}(L)=\mathscr{I}(\psi)
\] 
for another function induced by $h_{0}$. For any vector bundle $E$ we define
\[
\mathcal{H}^{p,q}_{\leqslant\lambda}(L^{k}\otimes E\otimes\mathfrak{I}(L^{k})):=\{\alpha\in\mathcal{H}^{p,q}_{\leqslant\lambda}(L^{k}\otimes E);\int_{X}|\alpha|^{2}e^{-\psi}<\infty\}.
\]

Then we prove that
\begin{theorem}\label{t1}
Let $L$ be a nef line bundle on $X$, and let $E$ be a vector bundle. Assume that $h_{0}$ has analytic singularities. Take $q\geqslant1$ and $m>0$. Then, if $0\leqslant\lambda\leqslant k$,
\begin{equation}\label{e1}
    h^{n,q}_{\leqslant\lambda}(L^{k}\otimes E\otimes\mathfrak{I}(L^{k}))\leqslant C(\lambda+1)^{q}k^{n-q}.
\end{equation}
If $1\leqslant k\leqslant\lambda$, then
\begin{equation}\label{e2}
    h^{n,q}_{\leqslant\lambda}(L^{k}\otimes E\otimes\mathfrak{I}(L^{k}))\leqslant Ck^{n}.
\end{equation}
\end{theorem}
Since $E$ is allowed to be an arbitrary vector bundle, we see by substituting $E\otimes\Omega^{p}_{X}\otimes K^{-1}_{X}$ for $E$, that the same asymptotic estimate also holds for the numbers $h^{p,q}_{\leqslant\lambda}$.

This kind of estimate has an important application in geometry, especially the $\lambda=0$ case. In fact, $\mathcal{H}^{n,q}_{\leqslant0}(X,L^{k}\otimes E)$ is just the space of the harmonic $L^{k}\otimes E$-valued $(n,q)$-forms on $X$, which is isomorphic to the Dolbeault cohomology group
\[
    H^{n,q}(X,L^{k}\otimes E).
\]
This isomorphism will be proved in next section (Proposition \ref{p21},\ref{p41}). It can be seen as a singular version of Hodge's theorem. In particular, $\mathcal{H}^{n,q}_{\leqslant0}(X,L^{k}\otimes E\otimes\mathfrak{I}(L^{k}))$ is just the image of the natural morphism
\[
i_{n,q}: H^{n,q}(X,L^{k}\otimes E\otimes\mathfrak{I}(L^{k}))\rightarrow H^{n,q}(X,L^{k}\otimes E\otimes\mathscr{I}(L^{k})).
\]
Based on this isomorphism, we eventually get an asymptotic estimate as follows.
\begin{theorem}\label{t2}
Let $L$ be a nef line bundle on $X$, and let $E$ be a vector bundle. Take $q\geqslant1$. We have the following conclusions:

1. Assume that $h_{0}$ is bounded, then
\[
h^{0,q}(L^{k}\otimes E)\leqslant Ck^{n-q}.
\]

2. Assume that $h_{0}$ has analytic singularities and $i\Theta_{L,h_{0}}$ has at least $n-s+1$ strictly positive eigenvalues at every point $x\in X$. Then
\[
   \dim\mathrm{Im}\mathrm{i}_{n,q}\leqslant Ck^{n-q}
\]
for $q\geqslant s$.

3. Assume that $X$ is K\"{a}hler, and $h_{0}$ has analytic singularities. Then
\[
    \dim\mathrm{Im}\mathrm{i}_{n,q}\leqslant Ck^{n-q}.
\]
\end{theorem}

The asymptotic estimate for the order of the cohomology group is a complicated problem in complex geometry. There are various work in this aspect. Here we only list a few of them. The first result is that
\[
    h^{0,q}(L^{k})\sim o(k^{n}),
\]
which is due to Siu \cite{Siu84a,Siu84b} when solving the Grauert--Riemenschneider conjecture \cite{GrR70}. Later Demailly also gives that
\[
    h^{0,q}(L^{k}\otimes E)\sim o(k^{n})
\]
for a nef line bundle $L$ and a vector bundle $E$ on a compact K\"{a}hler manifold based on his holomorphic Moser inequality \cite{Dem12}. Also we remark here that it can be optimized to
\[
h^{0,q}(L^{k}\otimes E)\sim O(k^{n-q})
\]
when $X$ is projective. Moreover, Matsumura \cite{Mat14,Mat15} generalizes it as
\[
    h^{0,q}(L^{k}\otimes E\otimes\mathscr{I}(h^{k}))\sim O(k^{n-q}),
\]
where $(L,h)$ is a pseudo-effective line bundle and $E$ is a vector bundle on a projective manifold. Here $\mathscr{I}(h)$ refers to the multiplier ideal sheaf. Recently, this result has been extended to a compact complex manifold by \cite{WaZ19} with additional requirement that $h$ has algebraic singularities. We remark here that \cite{WaZ19} also considers the asymptotic estimate for $h^{n,q}_{\leqslant\lambda}(L^{k}\otimes E)$ when $L$ is a semi-positive line bundle and $E$ is a vector bundle.

Our result also provides such type of estimate on a compact complex manifold, which is new.

Similar with \cite{Ber02}, we can use this estimate to solve the extension problem. As a result, we provide a generalization of the Grauert--Riemenschneider conjecture \cite{GrR70}.
\begin{theorem}[Generalization of the Grauert--Riemenschneider conjecture]\label{t3}
Let $X$ be a compact complex manifold, and let $L$ be a nef line bundle on $X$. Assume that one of the following situations occurs.

1. $h_{0}$ is bounded.

2. $h_{0}$ has analytic singularities and $i\Theta_{L,h_{0}}$ has at least $n-1$ strictly positive eigenvalues at every point $x\in X$. $i_{0,q}$ is either injective or surjective.

3. $X$ is K\"{a}hler, and $h_{0}$ has analytic singularities. $i_{0,q}$ is either injective or surjective. 

Then $L$ is big iff $(L)^{n}>0$.
\end{theorem}
The original Grauert--Riemenschneider conjecture says that if there is a semi-positive line bundle $L$ on a compact complex manifold $X$ such that the curvature $i\Theta_{L}>0$ on an open subset, $X$ must be Moishezon. It is well-known that $X$ is a Moishezon manifold iff there exists a big line bundle on $X$. So this conjecture actually says that if $i\Theta_{L}$ is semi-positive and strictly positive on an open subset, $L$ will be big. This conjecture has been solved by \cite{Dem85,Siu84a}. 

We also remark here that when $X$ is a projective manifold, the conclusion of Theorem \ref{t3} is a well-known result in algebraic geometry. Moreover, it has been extended to a K\"{a}hler manifold in \cite{DmP04} through the holomorphic Moser inequality. The method here is totally different.

Theorem \ref{t3} has good applications. We only mention tow of them. Firstly, it partially solves Demailly--P\u{a}un's conjecture posed in \cite{DmP04}.
\begin{conjecture}[Demailly--P\u{a}un]
Let $X$ be a compact complex manifold. Assume that $X$ possesses a nef cohomology class $[\alpha]$ of type $(1,1)$ such that $\int_{X}\alpha^{n}>0$.  Then $X$ is in the Fujiki class $\mathcal{C}$.
\end{conjecture}
Obviously, Theorem \ref{t3} partially confirms this conjecture when $\alpha$ is integral.

%Another application is the following consequence.
%\begin{corollary}\label{c11}
%Let $X$ be a compact K\"{a}hler manifold with nef tangent bundle. If there exists a pseudo-effective line bundle $(L,h)$ on $X$ such that $h$ has analytic singularities and $(L)^{n}>0$, $X$ is projective.
%\end{corollary}

Next, we shall give a Nadel-type vanishing theorem.
\begin{theorem}\label{t4}
Let $X$ be a compact K\"{a}hler manifold, and let $L$ be a nef line bundle. Assume that $h_{0}$ has analytic singularities and provides analytic Zariski decomposition. Let $m$ be the dimension of the pole-set of $h_{0}$. Then we have
\[
   H^{q}(X,K_{X}\otimes L\otimes\mathscr{I}(L))=0
\]
for $q>\max\{n-\kappa(L),m\}$.
\end{theorem}
The proof of theorem \ref{t4} is similar with the main result in \cite{Mat14}, but the conclusion here is independent. Also we use two examples (Sect. 2.4) to show that the requirement in Theorem \ref{t4} is not too demanding.

Notice that in \cite{Cao14}, it is shown another Nadel-type vanishing theorem saying that if $(L,\phi)$ is pseudo-effective, then
\[
   H^{q}(X,K_{X}\otimes L\otimes\mathscr{I}(\phi))=0
\]
for $q>n-\textrm{nd}(L,\phi)$. Here $\textrm{nd}(L,\phi)$ is the numerical dimension of $L$ associated with $\phi$ defined in \cite{Cao14}. Notice that we have
\[
\kappa(L)\leqslant\textrm{nd}(L),
\]
where $\textrm{nd}(L)$ is the numerical dimension of $L$ defined by intersection theory (of course without any specified metric). These two types of numerical dimension do not coincide well. Indeed the example in \cite{DPS94} shows that there do exists the case that $\textrm{nd}(L)>\textrm{nd}(L,\phi_{\min})$ with $\phi_{\min}$ the minimal singular metric on $L$. So it seems to me that there is no obvious relation between $\textrm{nd}(L,\phi)$ and $\kappa(L)$. Therefore it is not clear currently that whether the work in \cite{Cao14} implies Theorem \ref{t4}. Also we remark here that there is no obvious relationship between the work (Theorem 1.8) of \cite{WaZ19} and Theorem \ref{t4}.

In final, we make a discussion on the eigenform space
\[
\mathcal{H}^{n,q}_{\leqslant\lambda}(X,L)
\]
with $\lambda=0$ for a pseudo-effective line bundle $(L,h)$. More specifically, we will define $\mathcal{H}^{n,q}_{\leqslant0}(X,L)$ (Definition \ref{d51}, Sect.5) and prove a singular version of Hodge's theorem (Proposition \ref{p51}, Sect.5) when $L$ is merely pseudo-effective.

The plan of this paper is as follows. In Sect.2 we give a brief introduction on all the required materials including the nef line bundle, the modified ideal sheaf, Bergman kernel for the space $\mathcal{H}^{n,q}_{\leqslant\lambda}$, Siu's $\partial\bar{\partial}$-Bochner formula and so on. In Sect.3 we prove a submeanvalue inequality for forms in $\mathcal{H}^{n,q}_{\leqslant\lambda}$ and complete the proof of Theorem \ref{t1}. In Sect.4 we relate Theorem \ref{t1} to the asymptotic estimate for the cohomology group and give some applications. In the final section, we consider the $\lambda=0$ case for a pseudo-effective line bundle.

\begin{acknowledgment}
The author want to thank Prof. Bo Berndtsson, who introduced and carefully explained this problem to him. Also the author thanks Prof. Jixiang Fu for his suggestion and encouragement.
\end{acknowledgment}

\section{Preliminary}
\subsection{Nef line bundle}
Firstly, we briefly recall the multiplier ideal sheaf. Let $L\rightarrow X$ be a line bundle on a compact complex manifold $X$. Let $S:=\{h_{i}\}$ be a family of smooth metrics on $L$ with weight functions $\{\phi_{i}\}$, such that $\int_{X}e^{-\phi_{i}}\rightarrow\infty$ as $i$ tends to $\infty$, and $\int_{V}e^{-\phi_{i}}\leqslant C$ for some open subset $V$ of $X$. Then the Nadel-type multiplier ideal sheaf \cite{Nad90} (or dynamic multiplier ideal sheaf) at $x\in X$ can be defined as
\[
   \mathcal{I}(S)_{x}:=\{f\in\mathcal{O}_{X,x};\int_{U}|f|^{2}_{h_{i}}\leqslant C\textrm{ as }i\rightarrow\infty\},
\]
where $U$ is a local coordinate neighborhood of $x$.

On the other hand, if $h$ is a singular metric on $L$ with the weight function $\phi$, then its static multiplier ideal sheaf $\mathscr{I}(h)$ is defined in \cite{Dem12} by
\[
   \mathscr{I}(h)_{x}:=\{f\in\mathcal{O}_{X,x};|f|^{2}_{h}\textrm{ is integrable around }x\}.
\]

These two types of multiplier ideal sheaves coincide well when $L$ is a nef line bundle. Indeed, if $L$ is nef, by definition there exists a family of smooth metrics $S=\{h_{\varepsilon}\}$ such that $i\Theta_{L,h_{\varepsilon}}\geqslant-\varepsilon\omega$ for any $\varepsilon>0$. Here $\omega$ is a Hermitian metric on $X$ fixed before. Let $\phi_{\varepsilon}$ be the weight function of $h_{\varepsilon}$, then it is quasi-plurisubharmonic. Therefore $\{\phi_{\varepsilon}\}$ is locally bounded in $L^{1}$-norm, hence relatively compact. So we can find a subsequence $\{\phi_{\varepsilon_{i}}\}$ converging to a limit $\phi_{0}$ in $L^{1}$-norm. In particular, $i\Theta_{L,\phi_{0}}\geqslant0$. Let $h_{0}$ be the corresponding metric. In the rest part of this paper, $h_{0}$ will always refers to this metric if not specified. Now if $f\in\mathcal{I}(S)_{x}$, we have
\[
    \int_{U}|f|^{2}e^{-\phi_{0}}=\int_{U}\lim_{i\rightarrow\infty}|f|^{2}_{h_{\varepsilon_{i}}}=\lim_{i\rightarrow\infty}\int_{U}|f|^{2}_ {h_{\varepsilon_{i}}}<\infty
\]
by dominate convergence theorem. It means that $f\in\mathscr{I}(\phi_{0})_{x}$. On the other hand, if $g\in\mathscr{I}(\phi_{0})_{x}$, it is easy to see that $g\in\mathcal{I}(S)_{x}$ as well. In summary, we have
\[
   \mathcal{I}(S)=\mathscr{I}(\phi_{0})
\]
when $L$ is nef, and briefly denote it by $\mathcal{I}(L)$. It is also the start point of our work. For more information about the multiplier ideal sheaf (the dynamic one and the static one), one could refer to \cite{Dem12,Nad90}.

Next we shall present a canonical way to define the Laplacian operator associated to a nef line bundle $L$. First, by definition of the nef line bundle we have a family of smooth metrics $S=\{h_{\varepsilon}\}$ on $L$ with weight functions $\phi_{\varepsilon}$. We take its convergent subsequence and still denote it by $\{h_{\varepsilon}\}$. In particular, we have $h_{\varepsilon_{1}}\leqslant h_{\varepsilon_{2}}$ for any $0\leqslant\varepsilon_{2}\leqslant\varepsilon_{1}$, and the $L^{1}$-limit is denoted by $h_{0}$ with weight function $\phi_{0}$.

Fix a Hermitian metric $\omega$ on $X$. Since $h_{\varepsilon}$ is a smooth metric on $L$, we can define the Laplacian operator $\Delta_{\varepsilon}$ corresponds to $\omega$ and $h_{\varepsilon}$ in the usual sense. Now for any test $L$-valued $(p,q)$-form $\alpha$, we define the Laplacian operator associated to $h_{0}$ by
\[
    \Delta_{0}\alpha:=\lim_{\varepsilon\rightarrow0}\Delta_{\varepsilon}\alpha
\]
in the sense of $L^{2}$-topology. It is easy to verify that the limit exits if and only if $h_{\varepsilon}$ converges to $h_{0}$ in $L^{1}$-norm, while the later has been guaranteed. $\Delta_{0}$ possesses some basic properties of the classic Laplacian operator, such as $\Delta_{0}\bar{\partial}=\bar{\partial}\Delta_{0}$ and self-adjointness, i.e.
\[
<\Delta_{0}\alpha,\beta>_{h_{0}}=<\alpha,\Delta_{0}\beta>_{h_{0}}
\]
for any $\alpha,\beta$. It is just some basic calculation, so we omit the details here.

There is one issue to be concerned. $\Delta_{0}\alpha$ may not be a smooth form even if $\alpha$ is. So we need to carefully define the eigenvalue and eigenform. First, given two $L$-valued smooth $(p,q)$-forms $\alpha,\beta$ on $X$, we say that they are Dolbeault cohomological equivalent (it may not be a standard convention), if there exists an $L$-valued smooth $(p,q-1)$-form $\gamma$ such that $\alpha=\beta+\bar{\partial}\gamma$. It is easy to verify that it's an equivalence relationship. We briefly denote it by $\beta\in[\alpha]$ and vice versa. In particular, if $\alpha$ or $\beta$ is $\bar{\partial}$-closed, the Dolbeault cohomological equivalence just means that they belong to the same Dolbeault cohomology class. Now we have the following definition.
\begin{definition}\label{d21}
Let $\alpha$ be an $L$-valued $(p,q)$-form on $X$. Assume that for every $\varepsilon\ll1$, there exists a Dolbeault cohomological equivalent representative $\alpha_{\varepsilon}\in[\alpha]$ such that
\begin{enumerate}
  \item $\Delta_{\varepsilon}\alpha_{\varepsilon}=\mu\alpha_{\varepsilon}$. Here we ask that $\mu$ is independent of $\varepsilon$;
  \item $\alpha_{\varepsilon}\rightarrow\alpha$ in $L^{2}$-norm.
\end{enumerate}
Then we call $\alpha$ an eigenform of the Laplacian operator $\Delta_{0}$ with eigenvalue $\mu$. We simply denote it by $\Delta_{0}\alpha=\mu\alpha$.
\end{definition}
The eigenform space of $\Delta_{0}$ is defined as
\[
    \mathcal{H}^{p,q}_{\leqslant\lambda}(X,L,\Delta_{0}):=\{\alpha\in A^{p,q}(X,L);\Delta_{0}\alpha=\mu\alpha\textrm{ and }\mu\leqslant\lambda\}.
\]

We are especially interested in the $\lambda=0$ case since it corresponds to the Dolbeaut cohomology group. In fact, given a Dolbeault cohomology class $[\alpha]\in H^{p,q}(X,L)$, we have a unique $\Delta_{\varepsilon}$-harmonic representative $\alpha_{\varepsilon}\in[\alpha]$ for every $\varepsilon$ by Hodge's theorem. Moveover, since the harmonic representative minimizes the norm, we have $\|\alpha_{\varepsilon}\|_{h_{\varepsilon}}\leqslant\|\alpha\|_{h_{\varepsilon}}$. Assume that $\|\alpha\|_{h_{\varepsilon}}\leqslant C$ for all $\varepsilon$ (which means that $[\alpha]\in H^{p,q}(X,L\otimes\mathcal{I}(L))$). Then we can find a convergent subsequence of $\{\alpha_{\varepsilon}\}$ with limit $\tilde{\alpha}$, and $\tilde{\alpha}\in[\alpha]$. Therefore $\tilde{\alpha}$ is an eigenform of the Laplacian operator $\Delta_{0}$ with eigenvalue $0$ by definition. In other word, we could say that $\tilde{\alpha}$ is $\Delta_{0}$-harmonic. We remark here that in general $\tilde{\alpha}$ is merely an $L^{2}$-bounded $(p,q)$-form. But if $p=n$, $\tilde{\alpha}$ must be smooth. Indeed, it is not hard to see that
\[
\tilde{\alpha}=c_{n-q}(\lim_{\varepsilon\rightarrow0}\ast\alpha_{\varepsilon})\wedge\omega_{q}.
\]
Using the Kodaira--Akizuki--Nakano formula, $\lim_{\varepsilon\rightarrow0}\ast\alpha_{\varepsilon}$ is a $\bar{\partial}$-closed $(n-q,0)$-form hence holomorphic. It is enough for our purpose.

On the other hand, a $\Delta_{0}$-harmonic form $\alpha$ must be $\bar{\partial}$-closed by definition, so it naturally defines a cohomology class
\[
[\alpha]\in H^{p,q}(X,L).
\]
Obviously, if $\alpha,\beta$ are two different $\Delta_{0}$-harmonic forms, $[\alpha]\neq[\beta]$ in $H^{p,q}(X,L)$. In summary we have eventually proved a singular version of Hodge's theorem.
\begin{proposition}[A singular version of Hodge's theorem, I]\label{p21}
Let $X$ be a compact complex manifold, and let $L$ be a nef line bundle on $X$. Let $\Delta_{0}$ be the Laplacian operator defined before. Then we have
\begin{equation}\label{e3}
\begin{split}
   \mathcal{H}^{n,q}_{\leqslant0}(X,L,\Delta_{0})&\subset H^{n,q}(X,L),\\
   \mathcal{H}^{n,q}_{\leqslant0}(X,L\otimes\mathcal{I}(L),\Delta_{0})&\simeq H^{n,q}(X,L\otimes\mathcal{I}(L)).
\end{split}
\end{equation}
\begin{proof}
Based on the discussions above, it only remains to show that $\lim_{\varepsilon\rightarrow0}\ast\alpha_{\varepsilon}$ is $\bar{\partial}$-closed. Let's recall the Kodaira--Akizuki--Nakano formula, which says
\[
\|\bar{\partial}\alpha\|^{2}_{h}+\|\bar{\partial}^{\ast}\alpha\|^{2}_{h}=\|(\partial^{h})^{\ast}\alpha\|^{2}_{h}+\|\partial^{h}\alpha\|^{2}_{h}+ (i[\Theta_{L,h},\Lambda]\alpha,\alpha)_{h}
\]
for any $\alpha\in A^{n,q}(X,L)$ and smooth metric $h$ on $L$. Let $h=h_{\varepsilon}$ and $\alpha=\alpha_{\varepsilon}$, we have
\[
\|\bar{\partial}\alpha_{\varepsilon}\|^{2}_{h_{\varepsilon}}+\|\bar{\partial}^{\ast}\alpha_{\varepsilon}\|^{2}_{h_{\varepsilon}}
=\|(\partial^{h_{\varepsilon}})^{\ast}\alpha_{\varepsilon}\|^{2}_{h_{\varepsilon}}+\|\partial^{h_{\varepsilon}}\alpha_{\varepsilon}\|^{2}_{h _{\varepsilon}}+(i[\Theta_{L,h_{\varepsilon}},\Lambda]\alpha_{\varepsilon},\alpha_{\varepsilon})_{h_{\varepsilon}}.
\]
Since $\alpha_{\varepsilon}$ is $\Delta_{\varepsilon}$-harmonic, the left hand is zero. Take the limit with respect to $\varepsilon$, we get that
\[
0=\lim_{\varepsilon\rightarrow0}(\|(\partial^{h_{\varepsilon}})^{\ast}\alpha_{\varepsilon}\|^{2}_{h_{\varepsilon}}+\|\partial^{h_{\varepsilon}} \alpha_{\varepsilon}\|^{2}_{h _{\varepsilon}})+(i[\Theta_{L,h_{0}},\Lambda]\alpha_{0},\alpha_{0})_{h_{0}}).
\]
Since $\|(\partial^{h_{\varepsilon}})^{\ast}\alpha_{\varepsilon}\|^{2}_{h_{\varepsilon}}$, $\|\partial^{h_{\varepsilon}} \alpha_{\varepsilon}\|^{2}_{h _{\varepsilon}}$ and $i\Theta_{L,h_{0}}$ are non-negative, we eventually get that
\[
\lim_{\varepsilon\rightarrow0}\|(\partial^{h_{\varepsilon}})^{\ast}\alpha_{\varepsilon}\|^{2}_{h_{\varepsilon}}=\lim_{\varepsilon\rightarrow0} \|\partial^{h_{\varepsilon}} \alpha_{\varepsilon}\|^{2}_{h _{\varepsilon}}=(i[\Theta_{L,h_{0}},\Lambda]\alpha_{0},\alpha_{0})_{h_{0}})=0.
\]
In particular,
\[
\lim_{\varepsilon\rightarrow0}(\partial^{h_{\varepsilon}})^{\ast}\alpha_{\varepsilon}=\ast\bar{\partial}(\lim_{\varepsilon\rightarrow0}\ast \alpha_{\varepsilon})=0.
\]
It exactly implies that $\lim_{\varepsilon\rightarrow0}\ast\alpha_{\varepsilon}$ is $\bar{\partial}$-closed.
\end{proof}
\end{proposition}
The proof of Proposition \ref{p21} also shows that the $\Delta_{0}$-harmonic representative minimizes the $L^{2}$-norm defined by $h_{0}$.

As is shown before, an element of $\mathcal{H}^{n,q}_{\leqslant0}(X,L,\Delta_{0})$ must be $\bar{\partial}$-closed. For a general $\lambda$, we will see (in the proof of the main result) that the $\alpha\in\mathcal{H}^{n,q}_{\leqslant\lambda}(X,L,\Delta_{0})$ with $\bar{\partial}\alpha=0$ also plays an important role in the estimate of the number $h^{n,q}_{\leqslant\lambda}$.

When $X$ is K\"{a}hler, one could even parallel extend the other properties in Hodge theory to this situation. However, it is not the theme of this paper, so we will leave it for the future.

\subsection{The modified ideal sheaf}
We will introduce a notion called the modified ideal sheaf in this subsection. Remember that for a singular metric $\varphi$ on $L$ with analytic singularities, its (static) multiplier ideal sheaf can be computed precisely. Indeed, suppose that
\[
\varphi\sim a\log(|f_{1}|^{2}+\cdots+|f_{N}|^{2})
\]
near the poles. Here $f_{i}$ is a holomorphic function. We define $\mathscr{S}$ to be the sheaf of holomorphic functions $h$ such that $|h|^{2}e^{-\frac{\varphi}{a}}\leqslant C$. Then one computes a smooth modification $\mu:\tilde{X}\rightarrow X$ of $X$ such that $\mu^{\ast}\mathscr{S}$ is an invertible sheaf $\mathcal{O}_{\tilde{X}}(-D)$ associated with a normal crossing divisor $D=\sum\lambda_{j}D_{j}$, where $D_{j}$ is the component of the exceptional divisor of $\tilde{X}$. Now, we have $K_{\tilde{X}}=\mu^{\ast}K_{X}+R$, where $R=\sum\rho_{j}D_{j}$ is the zero divisor of the Jacobian function of the blow-up map. After some simple computation shown in \cite{Dem12}, we will finally get that
\[
\mathscr{I}(\varphi)=\mu_{\ast}\mathcal{O}_{\tilde{X}}(\sum (\rho_{j}-\lfloor a\lambda_{j}\rfloor)D_{j}),
\]
where $\lfloor a\lambda_{j}\rfloor$ denotes the round down of the real number $a\lambda_{j}$.

Now we have the following definition.
\begin{definition}\label{d22}
Let $h$ be a singular metric on $L$ with weight function $\varphi$. Assume that $\varphi$ has analytic singularities. Fix the notations as before, the modified ideal sheaf is defined as
\[
\mathfrak{I}(\varphi):=\mu_{\ast}\mathcal{O}_{\tilde{X}}(\sum(-\lceil a\lambda_{j}\rceil)D_{j}).
\]
Here $\lceil a\lambda_{j}\rceil$ denotes the round up of the real number $a\lambda_{j}$.
\end{definition}

Let
\[
\tau_{j}=
\begin{cases}
\lambda_{j}+\frac{\rho_{j}}{a} & \textrm{if } a\lambda_{j} \textrm{ is an integer} \\
\lambda_{j}+\frac{\rho_{j}+1}{a} & \textrm{if } a\lambda_{j} \textrm{ is not an integer}.
\end{cases}
\]
Then
\[
\mu_{\ast}\mathcal{O}_{\tilde{X}}(\sum(-\lceil a\lambda_{j}\rceil)D_{j})=\mu_{\ast}\mathcal{O}_{\tilde{X}}(\sum(\rho_{j}-\lfloor a\tau_{j}\rfloor)D_{j}).
\]
Let $g_{j}$ be the generator of $D_{j}$ on a local coordinate ball $V_{k}$. We define a function $\psi_{k}=\mu_{\ast}(a\sum\tau_{j}\log(\sum|g_{j}|^{2}))$ on $\mu(V_{k})$. It is easy to verify that
\[
\mathfrak{I}(\varphi)=\mathscr{I}(\psi_{k}).
\]
Let $g^{j}_{ik}$ be the transition function of $\mathcal{O}_{\tilde{X}}(D_{j})$ between $V_{k}$ and $V_{i}$, then
\[
\psi_{k}=\psi_{i}+\mu_{\ast}(a\sum\tau_{j}\log(\sum|g^{j}_{ik}|^{2})).
\]
Since $g^{j}_{ik}$ is a nowhere vanishing holomorphic function,
\[
\mu_{\ast}(a\sum\tau_{j}\log(\sum|g^{j}_{ik}|^{2}))
\]
is bounded. So being $L^{2}$-bounded against $\psi_{k}$ is equivalent to be $L^{2}$-bounded against $\psi_{i}$ for any $i,k$. Then after gluing all the $\psi_{k}$ together to be $\psi$ via a partition of unity, we have
\[
\mathfrak{I}(\varphi)=\mathscr{I}(\psi).
\]
As a result, $\mathfrak{I}(\varphi)$ is an ideal sheaf.

We list a few basic properties of the modified ideal sheaf here.
\begin{proposition}\label{p22}
Let $\varphi$ be a singular metric with analytic singularities. Then
\begin{enumerate}
  \item $\mathfrak{I}(\varphi)\subset\mathscr{I}(\varphi)$, and $\mathfrak{I}(\varphi)=\mathscr{I}(\varphi)$ iff $\varphi$ is defined by a normal crossing divisor.
  \item If $f\in\mathfrak{I}(\varphi)$, $|f|^{2}e^{-\varphi}$ is bounded. More precisely,
  \[
  |f|^{2}e^{-\varphi}\sim\pi_{\ast}(\Pi_{j}|g_{j}|^{2(\lceil a\lambda_{j}\rceil-a\lambda_{j})})
  \]
  near the poles of $\varphi$. Remember that $g_{j}$ is the generator of $D_{j}$. Hence $|f|^{2}e^{-\varphi}$ actually vanishes near the poles of $\varphi$ unless $\varphi$ has algebraic singularities.
  \item $\mathfrak{I}(\varphi)=\mathcal{O}_{X}$ iff $\varphi$ is bounded.
\end{enumerate}
\begin{proof}
(1) The first assertion is obvious. Observe that $\mathfrak{I}(\varphi)=\mathscr{I}(\varphi)$ iff $\tau_{j}=\lambda_{j}$ for all $j$, iff $\rho_{j}=0$ and $a\lambda_{j}$ is an integer for all $j$. If so, $a$ must be a rational number. On the other hand, $\rho_{j}$ comes from the Jacobian divisor, $\rho_{j}=0$ means that the log-resolution $\mu$ is trivial, hence $\varphi$ is defined by a normal crossing divisor $E=\sum\sigma_{j}E_{j}$ on $X$. Namely,
\[
\varphi=a\sum\sigma_{j}|f_{j}|^{2}
\]
if we denote the generator of $E_{j}$ by $f_{j}$. Notice that $\varphi$ is a metric of $L$, we actually have $\mathcal{O}_{X}(E)=L$ and $a=1$.

(2) By definition,
\[
\mu^{\ast}(|f|^{2}e^{-\varphi})\sim\Pi_{j}|g_{j}|^{2(\lceil a\lambda_{j}\rceil-a\lambda_{j})}
\]
near the poles, hence the desired estimate. Moreover, $|f|^{2}e^{-\varphi}\neq0$ at the poles iff $\lceil a\lambda_{j}\rceil=a\lambda_{j}$. In this situation, $a$ is a rational number and that's the last assertion.

(3) is a direct consequence of (2).
\end{proof}
\end{proposition}
Let's pause for a second. We can exclude the situation that $\varphi$ has algebraic singularities here since it has been discussed in \cite{WaZ19}. So from now on, $|f|^{2}e^{-\varphi}$ will always vanish near the poles if $f\in\mathfrak{I}(\varphi)$.  

Note that although
\[
\mathfrak{I}(\varphi)=\mathscr{I}(\psi)
\]
for some function $\psi$, the modified ideal sheaf behaves differently from the static multiplier ideal sheaf in many aspects. The following superadditivity is an interesting evidence.
\begin{proposition}[Superadditivity]\label{p23}
Let $\varphi_{1},\varphi_{2}$ be two singular metrics with analytic singularities. Then
\[
\mathfrak{I}(\varphi_{1})\cdot\mathfrak{I}(\varphi_{2})\subset\mathfrak{I}(\varphi_{1}+\varphi_{2}).
\]
\begin{proof}
Suppose that
\[
\begin{split}
\varphi_{1}&\sim a\log(|f_{1}|^{2}+\cdots+|f_{N}|^{2}),
\end{split}
\]
and
\[
\begin{split}
\varphi_{2}&\sim b\log(|g_{1}|^{2}+\cdots+|g_{M}|^{2}),
\end{split}
\]
where $f_{i},g_{j}$ are holomorphic functions. We define $\mathscr{S},\mathscr{W}$ to be the sheaves of holomorphic functions $h,k$ such that
\[
|h|^{2}e^{-\frac{\varphi_{1}}{a}}\leqslant C\textrm{ and }|k|^{2}e^{-\frac{\varphi_{2}}{b}}\leqslant C
\]
respectively. Let $\mu:\tilde{X}\rightarrow X$ be a log-resolution such that $\mu^{\ast}\mathscr{S},\mu^{\ast}\mathscr{W}$ are invertible sheaves $\mathcal{O}_{\tilde{X}}(-D_{1}),\mathcal{O}_{\tilde{X}}(-D_{2})$ of the normal crossing divisors $D_{1}=\sum\lambda_{i}D_{i},D_{2}=\sum\tau_{i}D_{i}$, where $D_{j}$ is the component of the exceptional divisor of $\tilde{X}$. Thus we get that
\[
\begin{split}
\mathfrak{I}(\varphi_{1})&=\mu_{\ast}\mathcal{O}_{\tilde{X}}(\sum(-\lceil a\lambda_{i}\rceil)D_{i}),\\
\mathfrak{I}(\varphi_{2})&=\mu_{\ast}\mathcal{O}_{\tilde{X}}(\sum(-\lceil b\tau_{i}\rceil)D_{i}),
\end{split}
\]
and
\[
\mathfrak{I}(\varphi_{1}+\varphi_{2})=\mu_{\ast}\mathcal{O}_{\tilde{X}}(\sum(-\lceil a\lambda_{i}+b\tau_{i}\rceil)D_{i}).
\]
Then the conclusion follows easily from the fact that
\[
\lceil a\lambda_{i}\rceil+\lceil b\tau_{i}\rceil\geqslant\lceil a\lambda_{i}+b\tau_{i}\rceil.
\]
\end{proof}
\end{proposition}

We prove a Koll\'{a}r-type injectivity theorem to finish this subsection. The more discussion about the modified ideal sheaf can be found in our papers \cite{Wu20}.

Remember that $\psi$ is a function on $X$ induced by $h_{0}$. Let 
\[
\mathcal{H}^{n,q}_{\leqslant0}(X,L\otimes\mathfrak{I}(L)):=\{\alpha\in\mathcal{H}^{n,q}_{\leqslant0}(X,L);\int_{X}|\alpha|^{2}e^{-\psi}<\infty\}.
\]

\begin{proposition}[A Koll\'{a}r-type injectivity theorem]\label{p25}
Assume that $X$ is a compact K\"{a}hler manifold. Let $s$ be a section of some multiple $L^{k-1}$ such that
\[
s\in H^{0}(X,L^{k-1}\otimes\mathfrak{I}(L^{k-1})).
\]
Then the following map
\[
\mathcal{H}^{n,q}_{\leqslant0}(X,L\otimes\mathfrak{I}(L))\xrightarrow{\otimes s}\mathcal{H}^{n,q}_{\leqslant0}(X,L^{k}\otimes\mathfrak{I}(L^{k}))
\]
induced by tensor with $s$ is injective.
\begin{proof}
By Proposition \ref{p21}, we have
\[
\begin{split}
\mathcal{H}^{n,q}_{\leqslant0}(X,L\otimes\mathfrak{I}(L))&\subset H^{n,q}(X,L\otimes\mathcal{I}(L))\\
\mathcal{H}^{n,q}_{\leqslant0}(X,L^{k}\otimes\mathfrak{I}(L^{k}))&\subset H^{n,q}(X,L^{k}\otimes\mathcal{I}(L^{k}))
\end{split}
\]
respectively.

By Gongyo--Matsumura's injectivity theorem \cite{GoM17},
\[
[\alpha]\in H^{n,q}(X,L\otimes\mathcal{I}(L))
\]
maps as $[s\alpha]$ injectively into $H^{n,q}(X,L^{k}\otimes\mathcal{I}(L^{k}))$. One verifies the section $s$ here must satisfy the condition in their theorem.

Now we consider the
\[
\alpha\in\mathcal{H}^{n,q}_{\leqslant0}(X,L\otimes\mathfrak{I}(L)).
\]
Since
\[
s\in H^{0}(X,L^{k-1}\otimes\mathfrak{I}(L^{k-1})),
\]
$s\alpha\in\mathcal{H}^{n,q}_{\leqslant0}(X,L^{k}\otimes\mathfrak{I}(L^{k}))$ by superadditivity (Proposition \ref{p23}). The proof is finished.
\end{proof}
\end{proposition}

One refers to \cite{Fuj12,Ko86a,Ko86b,Mat15} for the history of Koll\'{a}r's injectivity theorem.
\subsection{Bergman kernel for the space $\mathcal{H}^{n,q}_{\leqslant\lambda}$}
The estimate of the numbers $h^{n,q}_{\leqslant\lambda}$ is based on an observation about the Bergman kernel. The Bergman kernel at $x\in X$ is defined as the function
\[
   B(x)=\sum|\alpha_{j}(x)|^{2},
\]
where $\{\alpha_{j}\}$ is an orthonormal basis for $\mathcal{H}^{n,q}_{\leqslant\lambda}$, and the norm is the pointwise norm defined by the metrics $h_{0}$ and $\omega$ on $L$ and $X$. More precisely, $B(x)$ is the pointwise trace on the diagonal of the true Bergman kernel, defined as the reproducing kernel for $\mathcal{H}^{n,q}_{\leqslant\lambda}$.

The relevance of $B(x)$ for our problem lies in the formula
\[
   \int_{X}B(x)=h^{n,q}_{\leqslant\lambda},
\]
which is evident since each term in the definition of $B(x)$ contributes a $1$ to the integral. On the other hand, $B(x)$ is intimately related to the solution of the extremal problem
\[
   S(x)=\frac{\sup|\alpha(x)|^{2}}{\|\alpha\|^{2}},
\]
where the supremum is taken over all $\alpha$ in $\mathcal{H}^{n,q}_{\leqslant\lambda}$. Indeed, the following lemma is classical in Bergman's theory of reproducing kernels. Let $E$ be a Hermitian vector bundle of rank $N$ on a manifold $X$. Let $V$ be a subspace of the space of continuous global sections of $E$ whose coefficients are in $L^{2}(X)$, and let $\{\alpha_{j}\}$ be an orthonormal basis for $V$. Define $B(x)$ and $S(x)$ with $\{\alpha_{j}\}$ and space $V$ same as before. Then we have
\begin{lemma}\label{l22}
\[
   S(x)\leqslant B(x)\leqslant NS(x).
\]
In particular,
\[
   \int_{X}S(x)\leqslant\dim(V)\leqslant N\int_{X}S(x).
\]
\end{lemma}
The proof can be found in \cite{Ber02}.

Theorem \ref{t1} therefore follows if we can prove a submeanvalue inequality that estimates the value of a form $\alpha\in\mathcal{H}^{n,q}_{\leqslant\lambda}$ at any point $x\in X$ by its $L^{2}$-norm.

\subsection{Siu's $\partial\bar{\partial}$-Bochner formula}
The $\partial\bar{\partial}$-Bochner formula for an $L$-valued $(n,q)$-form is first developed by Siu in \cite{Siu82} on a compact K\"{a}hler manifold, then it is extended to a general compact complex manifold in \cite{Ber02}. Furthermore, it is generalized in \cite{WaZ19} to a version suitable for a line bundle $L$ tensoring with a vector bundle $E$. For our purpose, we only present the latest version in \cite{WaZ19} here.

\begin{proposition}\label{p27}
Let $(X,\omega)$ be a compact complex manifold. Let $E$ and $L$ be holomorphic vector bundle of rank $r$ and line bundle respectively. Let $\alpha$ be an $L\otimes E$-valued $(n,q)$-form. If $\alpha$ is $\bar{\partial}$-closed, the following inequality holds:
\begin{equation}\label{e4}
   i\partial\bar{\partial}T_{\alpha}\wedge\omega_{q-1}\geqslant(-2\textrm{Re}<\Delta\alpha,\alpha>+<i\Theta_{L\otimes E}\wedge\Lambda\alpha,\alpha>-c|\alpha|^{2})\omega_{n}.
\end{equation}
The constant $c$ is zero if $\bar{\partial}\omega_{q-1}=\bar{\partial}\omega_{q}=0$. Here $T_{\alpha}=c_{n-q}\ast\alpha\wedge\ast\bar{\alpha}e^{-\phi}$.
\end{proposition}
The proof can be found in \cite{WaZ19}.

If we denote $\gamma=\ast\alpha$, $\alpha$ can be expressed as
\[
   \alpha=c_{n-q}\gamma\wedge\omega_{q},
\]
and we moreover have
\[
    \ast\gamma=(-1)^{n-q}c_{n-q}\gamma\wedge\omega_{q}.
\]
Then $|\alpha|^{2}\omega_{n}=T_{\alpha}\wedge\omega_{q}$, so the norm of $\alpha$ is given by the trace of $T_{\alpha}$.
\subsection{Two canonical metrics}
In this subsection, we present two examples of singular metrics that provide analytic Zariski decomposition for the modified ideal sheaf. Let $L$ be a line bundle on a compact complex manifold $X$. In \cite{Siu98}, Siu introduces a special singular metric $\phi_{\textrm{siu}}$ as follows. For a basis $\{s^{k}_{j}\}^{N_{k}}_{j=1}$ of $H^{0}(X,L^{k})$, the metric $\phi_{k}$ is defined by
\[
   \phi_{k}:=\frac{1}{k}\log\sum^{N_{k}}_{j=1}|s^{k}_{j}|^{2}.
\]
Take a convergent sequence $\{\varepsilon_{k}\}$, and the Siu-type metric $\phi_{\textrm{siu}}$ on $L$ is then defined by
\[
   \phi_{\textrm{siu}}:=\log\sum^{\infty}_{k=1}\varepsilon_{k}e^{\phi_{k}}.
\]
Certainly $\phi_{\textrm{siu}}$ is pseudo-effective and provides an analytic Zariski decomposition.
\begin{proposition}[Analytic Zariski decomposition]\label{p28}
For all $k\geqslant0$, we have
\[
\begin{split}
    H^{0}(X,L^{k}\otimes\mathfrak{I}(h^{k}_{\mathrm{siu}}))=H^{0}(X,L^{k}\otimes\mathscr{I}(h^{k}_{\mathrm{siu}}))=H^{0}(X,L^{k}).
\end{split}
\]
\end{proposition}

Apart from $h_{\textrm{siu}}$, one could also consider the metric $h_{\textrm{min}}$ with minimal singularity.
\begin{definition}\label{d23}
Let $L$ be a pseudo-effective line bundle. Consider two Hermitian metrics $h_{1},h_{2}$ on $L$ with curvature $i\Theta_{L,h_{j}}\geqslant0$ in the sense of currents.

1. We will write $h_{1}\preceq h_{2}$, and say that $h_{1}$ is less singular than $h_{2}$, if there exits a constant $C>0$ such that $h_{1}\leqslant Ch_{2}$.

2. We will write $h_{1}\sim h_{2}$, and say that $h_{1},h_{2}$ are equivalent with respect to singularities, if there exists a constant $C>0$ such that $C^{-1}h_{2}\leqslant h_{1}\leqslant Ch_{2}$.
\end{definition}

The above definition is motivated by the following observation.
\begin{lemma}(\cite{DPS01})\label{l23}
For every pseudo-effective line bundle $L$, there exits up to equivalence of singularities a unique Hermitian metric $h_{\textrm{min}}$ with minimal singularities such that $i\Theta_{L,h_{\textrm{min}}}\geqslant0$.
\end{lemma}

Certainly we have
\[
H^{0}(X,L^{k}\otimes\mathscr{I}(h^{k}_{\textrm{min}}))=H^{0}(X,L^{k})\textrm{ for all }k\geqslant0.
\]

Moreover, if $h_{\textrm{min}}$ has analytic singularities, we have
\[
H^{0}(X,L^{k}\otimes\mathfrak{I}(h^{k}_{\textrm{min}}))=H^{0}(X,L^{k})\textrm{ for all }k\geqslant0.
\]

\section{A submeanvalue inequality for the $\Delta_{0}$-eigenforms and the estimate of $h^{n,q}_{\leqslant\lambda}$}
This section is devoted to prove a submeanvalue inequality. The argument here is borrowed from \cite{Ber02} with necessary adjustment. Here and in the later part of this paper, $L$ is assumed to be a nef line bundle on a compact complex manifold $X$ unless specified, and
\[
\mathcal{H}^{p,q}_{\leqslant\lambda}(X,L\otimes E,\Delta_{0})
\]
is the eigenform space defined in Sect.2, which is briefly denoted by $\mathcal{H}^{p,q}_{\leqslant\lambda}(X,L\otimes E)$. $E$ is a vector bundle.

Fix a point $x$ in $X$ and choose local coordinates, $z=(z_{1},...,z_{n})$ near $x$ such that $z(x)=0$ and the metric form $\omega$ on $X$ satisfying
\[
    \omega=\frac{i}{2}\partial\bar{\partial}|z|^{2}=:\beta
\]
at the point $x$. The next proposition is the crucial step in this argument.
\begin{proposition}\label{p31}
With the same notations as in Sect.2, let
\[
\alpha\in\mathcal{H}^{n,q}_{\leqslant\lambda}(X,L^{k}\otimes E\otimes\mathcal{I}(L^{k}))
\]
satisfy $\bar{\partial}\alpha=0$. Then for $r<\lambda^{-1/2}$ and $r<c_{0}$,
\[
    \int_{|z|<r}|\alpha|^{2}_{h^{k}_{0}}\leqslant Cr^{2q}(\lambda+1)^{q}\int_{X}|\alpha|^{2}_{h^{k}_{0}}.
\]
The constants $c_{0}$ and $C$ are independent of $k$, $\lambda$ and the point $x$.
\begin{proof}
Since $L$ is nef, there exists a family of smooth metrics $\{h_{\varepsilon}\}$ such that $i\Theta_{L,h_{\varepsilon}}\geqslant-\varepsilon\omega$ for every $\varepsilon>0$. Moreover, the limit of (a subsequence of) $\{h_{\varepsilon}\}$ exits, and is denoted by $h_{0}$. We apply Proposition \ref{p27} to $(L^{k},h^{k}_{\varepsilon})$. The expression $<i\Theta_{L^{k}\otimes E}\wedge\Lambda\alpha,\alpha>\omega_{n}$ can be estimated from below by a constant $c(1-k\varepsilon)$ times $|\alpha|^{2}_{h^{k}_{\varepsilon}}$, so we get
\begin{equation}\label{e5}
    i\partial\bar{\partial}T_{\alpha,\varepsilon}\wedge\omega_{q-1}\geqslant(-2\textrm{Re}<\Delta_{\varepsilon}\alpha,\alpha>-c^{\prime} (1+k\varepsilon)|\alpha|^{2}_{h^{k}_{\varepsilon}})\omega_{n}.
\end{equation}
Here $T_{\alpha,\varepsilon}=c_{n-q}\ast\alpha\wedge\ast\bar{\alpha}e^{-\phi_{\varepsilon}}$, where $\phi_{\varepsilon}$ is the weight function of $h_{\varepsilon}$. For $r$ small, put
\[
    \sigma(r,\varepsilon)=\int_{|z|<r}|\alpha|^{2}_{h^{k}_{\varepsilon}}\omega_{n},
\]
then it is left to prove that
\[
    \sigma(r,0)\leqslant Cr^{2q}(\lambda+1)^{q}
\]
if we have normalized so that the $L^{2}$-norm of $\alpha$ with respect to $h_{0}$ is equal to $1$.

From (\ref{e5}) we see that
\begin{equation}\label{e6}
\begin{split}
    &\int_{|z|<r}(r^{2}-|z|^{2})i\partial\bar{\partial}T_{\alpha,\varepsilon}\wedge\omega_{q-1}\\
    \geqslant&-c^{\prime}(1+k\varepsilon)r^{2}\sigma(r,\varepsilon)- 2r^{2}\int_{|z|<r}|\Delta_{\varepsilon}\alpha|_{h^{k}_{\varepsilon}}|\alpha|_{h^{k}_{\varepsilon}}\omega_{n}.
\end{split}
\end{equation}

Put
\[
   \lambda(r,\varepsilon)=(\int_{|z|<r}|\Delta_{\varepsilon}\alpha|^{2}_{h^{k}_{\varepsilon}})^{1/2},
\]
and use Cauchy's inequality to obtain
\[
   \int_{|z|<r}|\Delta_{\varepsilon}\alpha|_{h^{k}_{\varepsilon}}|\alpha|_{h^{k}_{\varepsilon}}\omega_{n}\leqslant\lambda(r,\varepsilon) \sigma(r,\varepsilon)^{1/2}.
\]
Applying Stokes' formula to the left hand side of (\ref{e6}) we get, since $\beta=\frac{i}{2}\partial\bar{\partial}|z|^{2}$,
\begin{equation}\label{e7}
\begin{split}
   &2\int_{|z|<r}T_{\alpha,\varepsilon}\wedge\omega_{q-1}\wedge\beta\\
   \leqslant&\int_{|z|=r}T_{\alpha,\varepsilon}\wedge\omega_{q-1}\wedge\partial|z|^{2}+c^{\prime}(1+k\varepsilon)r^{2}\sigma(r,\varepsilon) +2r^{2}\sigma(r)^{1/2}\lambda(r,\varepsilon).
\end{split}
\end{equation}
Since $\omega$ is smooth, by the choice of local coordinates we have,
\[
   (1-O(r))\omega\leqslant\beta\leqslant(1-O(r))\omega.
\]
Hence
\[
   T_{\alpha,\varepsilon}\wedge\omega_{q-1}\wedge\beta\geqslant q(1-O(r))|\alpha|^{2}_{h^{k}_{\varepsilon}}\omega_{n}.
\]
Next, if $\omega=\beta$ the boundary integral in (\ref{e7}) can be estimated by an integral with respect to surface measure
\[
   r\int_{|z|=r}|\alpha|^{2}_{h^{k}_{\varepsilon}}dS,
\]
and this implies that in our case,
\[
   \int_{|z|=r}T_{\alpha,\varepsilon}\wedge\omega_{q-1}\wedge\partial|z|^{2}\leqslant r(1-O(r))\int_{|z|=r}|\alpha|^{2}_{h^{k}_{\varepsilon}}(\omega_{n}/\beta_{n})dS.
\]
But
\[
   \int_{|z|=r}|\alpha|^{2}_{h^{k}_{\varepsilon}}(\omega_{n}/\beta_{n})dS=\frac{d}{dr}\sigma(r,\varepsilon),
\]
so if we also incorporate the term $c^{\prime}r^{2}\sigma(r,\varepsilon)$ in $O(r)\sigma(r,\varepsilon)$, we get
\[
   2q(1-O(r))\sigma(r,\varepsilon)\leqslant r\frac{d}{dr}\sigma(r,\varepsilon)+2r^{2}\sigma(r,\varepsilon)^{1/2}\lambda(r,\varepsilon)+k\varepsilon r^{2}\sigma(r,\varepsilon).
\]
Notice the fact that $\alpha\in\mathcal{H}^{n,q}_{\leqslant\lambda}(X,L^{k}\otimes E\otimes\mathcal{I}(L^{k}))$,
\[
\begin{split}
\lim_{\varepsilon\rightarrow0}\sigma(r,\varepsilon)&=\sigma(r,0),\\
\lim_{\varepsilon\rightarrow0}\frac{d}{dr}\sigma(r,\varepsilon)&=\frac{d}{dr}\sigma(r,0),
\end{split}
\]
and
\[
\lim_{\varepsilon\rightarrow0}\lambda(r,\varepsilon)=\lambda(r,0).
\]
Therefore we have
\begin{equation}\label{e8}
   2q(1-O(r))\sigma(r,0)\leqslant r\frac{d}{dr}\sigma(r,0)+2r^{2}\sigma(r,0)^{1/2}\lambda(r,0)
\end{equation}
as $\varepsilon$ tends to zero. Then it follows the same analytic technique as in \cite{Ber02}, we conclude our desired result.
\end{proof}
\end{proposition}

With the help of this proposition, the problem is now reduced on a ball with radius $r$. It is left to estimate the weighted $L^{\infty}$-norm of
\begin{equation*}
\alpha\in\mathcal{H}^{n,q}_{\leqslant\lambda}(X,L^{k}\otimes E)
\end{equation*}
(i.e. $\sup_{|z|<r}|\alpha(x)|^{2}_{h^{k}_{0}}$) via its $L^{2}$-norm. However, the classic theory of the elliptic partial differential operator cannot help too much here, since in general $h_{0}$ may be infinite somewhere. Currently we can only prove the submeanvalue inequality for a nef line bundle such that $h_{0}$ has analytic singularities.
\begin{proposition}\label{p32}
Assume that $h_{0}$ has analytic singularities. Let
\[
\alpha\in\mathcal{H}^{n,q}_{\leqslant\lambda}(X,L^{k}\otimes E\otimes\mathfrak{I}(L^{k}))
\]
satisfy $\bar{\partial}\alpha=0$. Then for any $x\in X$
\[
   |\alpha(x)|^{2}_{h^{k}_{0}}\leqslant Ck^{n-q}(\lambda+1)^{q}\int_{X}|\alpha|^{2}_{h^{k}_{0}}\omega_{n}
\]
if $\lambda\leqslant k$ and
\[
   |\alpha(x)|^{2}_{h^{k}_{0}}\leqslant C\lambda^{n}\int_{X}|\alpha|^{2}_{h^{k}_{0}}\omega_{n}
\]
if $\lambda\geqslant k\geqslant1$. The constant is independent of $k$, $\lambda$ and $x$.
\begin{proof}
The proof is mostly borrowed from \cite{Ber02}, which is a clever application of the localization technique. Assume first $\lambda\leqslant k$ and fix $x\in X$. Choose as before local coordinates, $z$, near $x$ such that $z(x)=0$ and $\omega=\frac{i}{2}\partial\bar{\partial}|z|^{2}=\beta$ at the point $x$. Choose also local trivializations of $L$ and $E$ near $x$. Now we take a family of metrics $\{h_{\varepsilon}\}$ with weight functions $\phi_{\varepsilon}$ on $L$ as before. For any $\varepsilon>0$, we may assume the local trivialization is chosen so that the metric $\phi_{\varepsilon}$ on $L$ has the form
\[
   \phi_{\varepsilon}=\sum\mu_{j}|z_{j}|^{2}+o(|z|^{2}).
\]
For any $\alpha$ we express it in terms of the trivialization and local coordinates and put
\[
    \alpha^{(k)}(z)=\alpha(z/\sqrt{k}),
\]
so that $\alpha^{(k)}$ is defined for $|z|<1$ if $k$ is large enough. We also scale the Laplacian by putting
\[
    k\Delta^{(k)}_{\varepsilon}\alpha^{(k)}=(\Delta_{\varepsilon}\alpha)^{(k)}.
\]
It is not hard to see that if $\Delta$ is defined by the metric $\psi$ on $F_{k}$, then $\Delta^{(k)}$ is associated to the line bundle metric $\psi(z/\sqrt{k})$. In particular, if $F_{k}=L^{k}$ and $\psi=k\phi_{\varepsilon}$, then $\Delta^{(k)}_{\varepsilon}$ is associated to
\[
    \sum\mu_{j}|z_{j}|^{2}+o(1),
\]
and hence converges to a $k$-independent elliptic operator. Obviously, the same thing happens even if we substitute $L^{k}$ by $L^{k}\otimes E$ for a vector bundle $E$. It therefore follows from G{\aa}rding's inequality together with Sobolev estimates that
\[
    |\alpha(0)|^{2}_{h^{k}_{\varepsilon}}\leqslant C(\int_{|z|<1}|\alpha^{(k)}|^{2}_{h^{k}_{\varepsilon}}\omega_{n}+\int_{|z|<1}|(\Delta^{(k)}_{\varepsilon})^{m}\alpha^{(k)}|^{2}_{h^{k}_{\varepsilon}} \omega_{n}),
\]
if $m>n/2$. Remember that when $0\in\{h_{0}=\infty\}$, $|\alpha(0)|^{2}_{h^{k}_{\varepsilon}}=0$, we can find a $C_{1}$ independent of $\varepsilon$ such that
\[
|\alpha(0)|^{2}_{h^{k}_{\varepsilon}}\leqslant C_{1}(\int_{|z|<1}|\alpha^{(k)}|^{2}_{h^{k}_{\varepsilon}}\omega_{n}+\int_{|z|<1}|(\Delta^{(k)}_{\varepsilon})^{m}\alpha^{(k)}|^{2}_{h^{k}_{\varepsilon}} \omega_{n})
\]
On the other hand, when $0\notin\{h_{0}=\infty\}$, $h_{0}$ is a smooth metric at 0. Therefore using G{\aa}rding's inequality together with Sobolev estimates against $h_{0}$, we can find a $C_{2}$ such that
\[
|\alpha(0)|^{2}_{h^{k}_{0}}\leqslant C_{2}(\int_{|z|<1}|\alpha^{(k)}|^{2}_{h^{k}_{0}}\omega_{n}+\int_{|z|<1}|(\Delta^{(k)}_{0})^{m}\alpha^{(k)}|^{2}_{h^{k}_{0}}\omega_{n}).
\]

Then combine the two estimates above together, we have
\begin{equation}\label{e9}
|\alpha(0)|^{2}_{h^{k}_{0}}\leqslant C_{3}(\int_{|z|<1}|\alpha^{(k)}|^{2}_{h^{k}_{0}}\omega_{n}+\int_{|z|<1}|(\Delta^{(k)}_{0})^{m}\alpha^{(k)}|^{2}_{h^{k}_{0}} \omega_{n})
\end{equation}
Now
\[
    \int_{|z|<1}|\alpha^{(k)}|^{2}_{h^{k}_{0}}\omega_{n}=k^{n}\int_{|z|<1/\sqrt{k}}|\alpha|^{2}_{h^{k}_{0}}\omega_{n}
\]
and
\[
    \int_{|z|<1}|\Delta^{(k)}_{0}\alpha^{(k)}|^{2}_{h^{k}_{0}}\omega_{n}=k^{n-2m}\int_{|z|<1/\sqrt{k}}|(\Delta_{0})^{m} \alpha|^{2}_{h^{k}_{0}}\omega_{n}.
\]
As a result,
\begin{equation}\label{e10}
    |\alpha(0)|^{2}_{h^{k}_{0}}\leqslant C_{3}(k^{n}\int_{|z|<1/\sqrt{k}}|\alpha|^{2}_{h^{k}_{0}}\omega_{n}+k^{n-2m}\int_{|z|<1/\sqrt{k}}|(\Delta_{0})^{m} \alpha|^{2}_{h^{k}_{0}}\omega_{n}).
\end{equation}
Do the normalization so that the $L^{2}$-norm of $\alpha$ with respect to $h_{0}$ is one. By Proposition \ref{p22}, (1) and Proposition \ref{p31} we have
\[
    k^{n}\int_{|z|<1/\sqrt{k}}|\alpha|^{2}_{h^{k}_{0}}\omega_{n}\leqslant Ck^{n-q}(\lambda+1)^{q},
\]
and
\[
\begin{split}
    k^{n-2m}\int_{|z|<1/\sqrt{k}}|(\Delta_{0})^{m}\alpha|^{2}_{h^{k}_{0}}\omega_{n}&\leqslant Ck^{n-q}(\lambda+1)^{q}(\lambda/k)^{2m}\\
                                                                               &\leqslant Ck^{n-q}(\lambda+1)^{q}.
\end{split}
\]
Combine these two inequalities with (\ref{e10}), we have thus proved the first part of Proposition \ref{p32}. The second statement is much easier. We now apply (\ref{e10}) to the scaling $\alpha^{(\lambda)}$ instead, and get immediately that
\[
    |\alpha(0)|^{2}_{h^{k}_{0}}\leqslant C\lambda^{n}.
\]
\end{proof}
\end{proposition}

We now have all the ingredients for the proof of Theorem \ref{t1}.
\begin{proof}[Proof of Theorem \ref{t1}]
Let first $\mathcal{Z}^{n,q}_{\leqslant\lambda}$ be the subspace of $\mathcal{H}^{n,q}_{\leqslant\lambda}$ consisting of all the $\bar{\partial}$-closed forms. We apply Lemma \ref{l22} with $L^{k}\otimes E$-valued $(n,q)$-forms. The estimate for $S(x)$ furnished by Proposition \ref{p32} together with Proposition \ref{l22} then immediately gives Theorem \ref{t1} for $\mathcal{Z}^{n,q}_{\leqslant\lambda}$. We now claim that
\begin{equation}\label{e11}
   h^{n,q}_{\leqslant\lambda}\leqslant\dim\mathcal{Z}^{n,q}_{\leqslant\lambda}+\dim\mathcal{Z}^{n,q+1}_{\leqslant\lambda},
\end{equation}
which completes the proof since the estimate for $\dim\mathcal{Z}^{n,q+1}_{\leqslant\lambda}$ is better than our desired estimate for $h^{n,q}_{\leqslant\lambda}$.
The claim is not complicated and was first proved in \cite{Ber02}. We present here for readers' convenience. Let $\alpha$ is an eigenform of $\Delta_{0}$, so that
\[
   \Delta_{0}\alpha=\mu\alpha.
\]
If we decompose $\alpha=\alpha^{1}+\alpha^{2}$ where $\alpha^{1}$ is $\bar{\partial}$-closed and $\alpha^{2}$ is orthogonal (with respect to $h_{0}$) to the space of $\bar{\partial}$-closed forms, then the $\alpha^{j}$'s with $j=1,2$ are also eigenforms with the same eigenvalue. To see this, note that $\Delta_{0}$ commutes with $\bar{\partial}$, so $\bar{\partial}\Delta_{0}\alpha_{1}=0$ and
\[
   <\Delta_{0}\alpha^{2},\eta>_{h_{0}}=<\alpha^{2},\Delta_{0}\eta>_{h_{0}}=0
\]
if $\bar{\partial}\eta=0$. Hence
\[
   \Delta_{0}\alpha^{j}=(\Delta_{0}\alpha)^{j}=\mu\alpha^{j}
\]
for $j=1,2$.
Now we decompose
\[
   \mathcal{H}^{n,q}_{\leqslant\lambda}=\mathcal{Z}^{n,q}_{\leqslant\lambda}\oplus(\mathcal{H}^{n,q}_{\leqslant\lambda}\ominus\mathcal{Z}^{n,q} _{\leqslant\lambda}).
\]
Since $\bar{\partial}$ maps $\mathcal{H}^{n,q}_{\leqslant\lambda}\ominus\mathcal{Z}^{n,q} _{\leqslant\lambda}$ injectively into $\mathcal{Z}^{n,q+1} _{\leqslant\lambda}$, (\ref{e11}) follows and the proof of Theorem \ref{t1} is complete.
\end{proof}
Since a semi-positive line bundle will always satisfy the condition in Theorem \ref{t1}, we remark here that the example (Proposition 4.2) in \cite{Ber02} also shows that the order of magnitude given in Theorem \ref{t1} can not be improved in general.

\section{Applications in geometry}
In order to prove Theorem \ref{t2}, we first extend the singular Hodge's theorem to the modified ideal sheaf. Since $\mathfrak{I}(L)\subset\mathcal{I}(L)$, there exits a natural morphism
\[
i_{n,q}:H^{n,q}(X,L^{k}\otimes\mathfrak{I}(L^{k}))\rightarrow H^{n,q}(X,L^{k}\otimes\mathscr{I}(L^{k})).
\]
Then the Hodge's theorem can be stated as follows.
\begin{proposition}[A (weak) singular version of Hodge's theorem, II]\label{p41}
Let $X$ be a compact complex manifold, and let $L$ be a nef line bundle on $X$. Let $\Delta_{0}$ be the Laplacian operator defined before. Assume that $h_{0}$ has analytic singularities. We have the following conclusions:

1. Suppose $i\Theta_{L,h_{0}}$ has at least $n-s+1$ positive eigenvalues at every point $x\in X$. Then
\[
   \mathcal{H}^{n,q}_{\leqslant0}(X,L^{k}\otimes\mathfrak{I}(L^{k}))=\mathrm{Im}\mathrm{i}_{n,q}
\]
for $q\geqslant s$ and $k$ large enough. 

2. Suppose $X$ is K\"{a}hler. Then 
\[
   \mathcal{H}^{n,q}_{\leqslant0}(X,L^{k}\otimes\mathfrak{I}(L^{k}))=\mathrm{Im}\mathrm{i}_{n,q}
\]
for all $q$ and $k$.
\end{proposition}
We need to prove a lemma first.
\begin{lemma}\label{l41}
Assume that $i\Theta_{L,h_{0}}$ has at least $n-s+1$ positive eigenvalues at every point $x\in X$. If $\alpha\in H^{n,q}(X,L^{k}\otimes\mathfrak{I}(L))$ with $q\geqslant s$, then its $\Delta_{0}$-harmonic representative $\tilde{\alpha}$ satisfies that
\[
|\tilde{\alpha}|^{2}_{h_{0}}\leqslant|\alpha|^{2}_{h_{0}}\textrm{ near the poles}
\]
when $k\geqslant k_{0}$ for some $k_{0}$ large enough. In particular, if $X$ is K\"{a}hler, the assumption for $i\Theta_{L,h_{0}}$ is not needed and $k_{0}=1$.
\begin{proof}
Let $\varphi_{0}$ be the weight function of $h_{0}$, and let $Z$ be the pole-set of $\varphi_{0}$. We claim that there exits an open subset $V:=\{\phi<0\}$ defined by some quasi-plurisubharmonic function $\phi$ on $X$ with $Z\subset V$, satisfying that
\[
\int_{V}|\tilde{\alpha}|^{2}_{h_{0}}\leqslant\int_{V}|\alpha|^{2}_{h_{0}}.
\]
It leads to the conclusion. In fact, if $|\tilde{\alpha}|^{2}_{h_{0}}$ is bigger than $|\alpha|^{2}_{h_{0}}$ at some point $x\in V$, $|\tilde{\alpha}|^{2}_{h_{0}}>|\alpha|^{2}_{h_{0}}$ on an open subset of $V$ since $|\alpha|^{2}_{h_{0}}$ is uniformly bounded. Moreover, we can shrink $V$ by substituting $\phi+C$ for $\phi$. It is easy to get a contradiction to the integral inequality above when $V$ is small enough.

Now it remains to prove the claim. First we assume that $X$ is K\"{a}hler, and the proof is divided into four steps.

1. Let's recall two formulas. The first one is the generalized Kodaira--Akizuki--Nakano formula proved in \cite{Tak95}. Let $\psi$ be a smooth real-valued function on $X$, and let $h$ be a smooth metric on $L$. Then we have
\begin{equation}\label{e21}
\begin{split}
&\|\sqrt{\eta}(\bar{\partial}+\bar{\partial}\psi\wedge)\alpha\|^{2}_{h}+\|\sqrt{\eta}\bar{\partial}^{\ast}\alpha\|^{2}_{h}\\
=&\|\sqrt{\eta}(\partial^{h}-\partial\psi\wedge)^{\ast}\alpha\|^{2}_{h}+\|\sqrt{\eta}\partial^{h}\alpha\|^{2}_{h}+(i\eta[\Theta_{L,h}+\partial\bar {\partial}\psi,\Lambda]\alpha,\alpha)_{h}
\end{split}
\end{equation}
for any $\alpha\in A^{n,q}(X,L)$ and $\eta=e^{\psi}$.

The second formula can also be found in \cite{Tak95}:
\begin{equation}\label{e22}
\begin{split}
\|\sqrt{\eta}(\partial\psi\wedge)^{\ast}\alpha\|^{2}_{h}=\|\sqrt{\eta}\bar{\partial}\psi\wedge\alpha\|^{2}_{h}+\|\sqrt{\eta}(\bar{\partial}\psi\wedge) ^{\ast}\alpha\|^{2}_{h}.
\end{split}
\end{equation}

Now apply (\ref{e21}) with $\psi=1$ and $h=h_{\varepsilon}$, we have
\[
\begin{split}
\|\bar{\partial}^{\ast}\tilde{\alpha}\|^{2}_{h_{\varepsilon}}=\|(\partial^{h_{\varepsilon}})^{\ast}\tilde{\alpha}\|^{2}_{h_{\varepsilon}}+\|\partial^ {h_{\varepsilon}}\tilde{\alpha}\|^{2}_{h_{\varepsilon}}+(i[\Theta_{L,h_{\varepsilon}},\Lambda]\tilde{\alpha},\tilde{\alpha})_{h_{\varepsilon}}.
\end{split}
\]
Take the limit with respect to $\varepsilon$,
\[
\lim\|\bar{\partial}^{\ast}\tilde{\alpha}\|^{2}_{h_{\varepsilon}}=\lim\|\bar{\partial}^{\ast}\alpha_{\varepsilon}\|^{2}_{h_{\varepsilon}}=0,
\]
where $\alpha_{\varepsilon}$ is the harmonic representative of $\alpha$ with respect to $h_{\varepsilon}$. Moreover, since $i\Theta_{L,h_{0}}\geqslant0$, we have
\[
\lim\|(\partial^{h_{\varepsilon}})^{\ast}\tilde{\alpha}\|^{2}_{h_{\varepsilon}}=\lim\|\partial^ {h_{\varepsilon}}\tilde{\alpha}\|^{2}_{h_{\varepsilon}}=\lim(i[\Theta_{L,h_{\varepsilon}},\Lambda]\tilde{\alpha},\tilde{\alpha})_{h_{\varepsilon}}=0.
\]
2. We fix a smooth metric $\varphi_{\varepsilon}$, and take $\varphi_{\textrm{min},\varepsilon}$ be the metric with minimal singularity. Namely,
\[
i\Theta_{L,\varphi_{\varepsilon}}+i\partial\bar{\partial}\varphi_{\textrm{min},\varepsilon}\geqslant0\textrm{ and }\varphi_{\textrm{min},\varepsilon}\preceq\varphi_{0}.
\]
The definition for the singular metric with minimal singularity can be found in Sect.2.5. Notice that although the weight function $\varphi_{\textrm{min},\varepsilon}$ will depend on $\varepsilon$, the curvature current $i\Theta_{L,\varphi_{\textrm{min}}}:=i\Theta_{L,\varphi_{\varepsilon}}+i\partial\bar{\partial}\varphi_{\textrm{min},\varepsilon}$ is independent of $\varepsilon$ and positive. Take the limit with respect to $\varepsilon$, we finally get
\[
i\Theta_{L,\varphi_{\textrm{min}}}=i\Theta_{L,\varphi_{0}}+i\partial\bar{\partial}\phi_{1}
\]
for some $\phi_{1}$. Since $\varphi_{\textrm{min},\varepsilon}\preceq\varphi_{0}$, the pole-set of $\varphi_{0}$ is included in
\[
V_{1}:=\{-\phi_{1}<0\}.
\]
There is one issue here. If $\varphi_{\textrm{min},\varepsilon}\sim\varphi_{0}$, we should replace $\varphi_{\textrm{min}}$ by another singular metric $\sigma$. Namely, we consider a singular metric $\sigma$ such that $i\Theta_{L,\sigma}\geqslant0$ and $\sigma\succeq\varphi_{0}$ strictly. Accordingly, we will get
\[
i\Theta_{L,\sigma}=i\Theta_{L,\varphi_{0}}+i\partial\bar{\partial}\phi_{2}.
\]
At this time, the pole-set of $\varphi_{0}$ is included in $V_{2}:=\{\phi_{2}<0\}$.

3. Apply (\ref{e21}) again with $\psi=\phi_{i}$ for $i=1,2$ respectively. Take the limit with respect to $\varepsilon$, we get that
\[
\begin{split}
\|\sqrt{\eta}(\bar{\partial}\phi_{i}\wedge\tilde{\alpha})\|^{2}_{h_{0}}=\|\sqrt{\eta}(-\partial\phi_{i}\wedge)^{\ast}\tilde{\alpha}\|^{2}_{h_{0}}+ (i\eta[\Theta_{L,\varphi_{0}}+\partial\bar{\partial}\phi_{i},\Lambda]\tilde{\alpha}, \tilde{\alpha})_{h_{0}}.
\end{split}
\]
One may wonder whether (\ref{e21}) is suitable for a singular weight function $\phi_{i}$ here. Indeed, by \cite{DPS01} one can always approximate $\phi_{i}$ by a family of smooth metrics $\{\phi_{\delta}\}$. Apply (\ref{e21}) to $\phi_{\delta}$ and take the limit with respect to $\delta$, we will finally get our desired equality. Later, when applying formula (\ref{e22}) on a singular weight function, we use this technique again without pointing out.

Combine with (\ref{e22}), we obtain that
\[
(\bar{\partial}\phi_{i}\wedge)^{\ast}\tilde{\alpha}=0.
\]

4. Observe that if we restrict the arbitrary metrics $\omega,h$ on $X,L$ to an open subset $U=\{\chi<0\}$, and define the corresponding $L^{2}$-norm on $U$, we have
\begin{equation}\label{e23}
(\bar{\partial}\beta,\gamma)=(\beta,\bar{\partial}^{\ast}\gamma)+[\beta,(\bar{\partial}\chi\wedge)^{\ast}\gamma].
\end{equation}
for any smooth $\alpha,\beta$. Setting $\tau:=dS/|d\chi|$ and $[\alpha,(\bar{\partial}\chi\wedge)^{\ast}\beta]$ is defined as
\[
[\alpha,(\bar{\partial}\chi\wedge)^{\ast}\beta]:=\int_{\partial U}(\alpha,(\bar{\partial}\chi\wedge)^{\ast}\beta)\tau.
\]
In particular, if $(\bar{\partial}\chi\wedge)^{\ast}\beta=0$ for any $\beta$, we have
\[
(\bar{\partial}\alpha,\beta)=(\alpha,\bar{\partial}^{\ast}\beta).
\]
Apply (\ref{e23}) with $h=h_{0}$, $\chi=\phi_{i}$ and $\gamma=\tilde{\alpha}$, we finally get that
\[
0=\lim_{\varepsilon\rightarrow0}\int_{V_{i}}(\bar{\partial}^{\ast}\tilde{\alpha},\beta)_{h_{\varepsilon}}=\int_{V_{i}}(\tilde{\alpha},\bar{\partial} \beta)_{h_{0}}
\]
for any $\beta$. We are ready to prove the inequality in the claim. Consider $\tilde{\alpha}+\bar{\partial}\beta$ for any $\beta$, we have
\[
\begin{split}
\int_{V_{i}}\|\tilde{\alpha}+\bar{\partial}\beta\|^{2}_{h_{0}}&=\int_{V_{i}}\|\tilde{\alpha}\|^{2}_{h_{0}}+\int_{V_{i}}\|\bar{\partial}\beta\|^{2}_ {h_{0}}+ 2\textrm{Re}\int_{V_{i}}(\tilde{\alpha},\bar{\partial}\beta)_{h_{0}}\\
&=\int_{V_{i}}\|\tilde{\alpha}\|^{2}_{h_{0}}+\int_{V_{i}}\|\bar{\partial}\beta\|^{2}_{h_{0}}.
\end{split}
\]
It means $\tilde{\alpha}$ minimizes the $L^{2}$-norm on $V_{i}$, hence the desired inequality. The proof of the claim is finished.

Now we deal with the second situation. Since $(X,\omega)$ is not necessary to be K\"{a}hler, we will use the Kodaira--Akizuki--Nakano formula for non-K\"{a}hler manifold proved in \cite{Dem85}.

Let $\tau$ be the operator of type $(1,0)$ defined by $\tau=[\Lambda,\partial\omega]$, and let
\[
\Delta_{\partial,\tau}=(\partial^{h}+\tau)(\partial^{h}+\tau)^{\ast}+(\partial^{h}+\tau)^{\ast}(\partial^{h}+\tau)
\]
be the $\partial$-Laplaican twisted by $\tau$. Then we have
\[
\Delta_{\bar{\partial}}=\Delta_{\partial,\tau}+[i\Theta_{L,h},\Lambda]+T_{\omega}.
\]
Here $T_{\omega}$ is an operator of order $0$ depending only on the torsion of the Hermitian metric $\omega$:
\[
T_{\omega}=[\Lambda,[\Lambda,\frac{i}{2}\partial\bar{\partial}\omega]]-[\partial\omega,(\partial\omega)^{\ast}].
\]

Use this formula to replace formula (\ref{e21}), then the same argument as before will lead to the desired inequality after we have shown that $[i\Theta_{L,h_{0}},\Lambda]+T_{\omega}$ is a positive operator.

So it is left to prove that $[i\Theta_{L,h_{0}},\Lambda]+T_{\omega}$ is positive. Since $i\Theta_{L,h_{0}}$ has at least $n-s+1$ positive eigenvalues, the computation in Theorem 5.1 of \cite{Dem09} shows that there exits a Hermitian metric $\omega_{\varepsilon}$ such that
\[
([i\Theta_{L,h},\Lambda_{\omega_{\varepsilon}}]\alpha,\alpha)_{\omega_{\varepsilon}}\geqslant(q-s+1-\varepsilon(s-1))|\alpha|^{2}
\]
for any $L$-valued $(n,q)$-form $\alpha$. Choosing $\varepsilon=1/s$ and $q\geqslant s$, the right hand side will be $\geqslant(1/s)|\alpha|^{2}$. Take $k$ large enough such that $[i\Theta_{L^{k},h_{0}},\Lambda]+T_{\omega}$ is positive, and the proof is complete.
\end{proof}
\end{lemma}
Next we prove Proposition \ref{p41}
\begin{proof}[Proof od Proposition \ref{p41}]
Let
\[
L^{n,q}_{(2)}(X,L)_{\varphi}
\]
be the space of all the $L$-valued $(n,q)$-form that is $L^{2}$-bounded against $\varphi$. $L^{n,q}_{(2)}(X,L)_{\psi}$ is similar. Let
\[
\textrm{Im}\bar{\partial}_{1}:=\textrm{Im}(\bar{\partial}:L^{n,q}_{(2)}(X,L)_{\varphi}\rightarrow L^{n,q}_{(2)}(X,L)_{\varphi}),
\]
and
\[
\textrm{Im}\bar{\partial}_{2}:=\textrm{Im}(\bar{\partial}:L^{n,q}_{(2)}(X,L)_{\psi}\rightarrow L^{n,q}_{(2)}(X,L)_{\psi}).
\]
$\textrm{Ker}\bar{\partial}_{1}$ and $\textrm{Ker}\bar{\partial}_{2}$ are defined similarly. Recall that there are following orthogonal decompositions \cite{GHS98}:
\[
\textrm{Ker}\bar{\partial}_{1}=\textrm{Im}\bar{\partial}_{1}\bigoplus\mathcal{H}^{n,q}_{\leqslant0}(X,L\otimes\mathscr{I}(L),\Delta_{0})
\]
and
\[
\textrm{Ker}\bar{\partial}_{2}=\textrm{Im}\bar{\partial}_{2}\bigoplus(\textrm{Ker}\bar{\partial}_{2}\cap\textrm{Ker}\bar{\partial}^{\ast}_{\psi}).
\]
By $\bar{\partial}^{\ast}_{\psi}$ we refer to the formal adjoint operator of $\bar{\partial}$ with respect to the $L^{2}$-norm defined by $\varphi$. One may wonder that are these decompositions still valid for singular metrics. Indeed, we can approximate them by smooth metrics then take the limit. On the other hand, it is easy to verify that
\begin{equation}\label{e24}
\textrm{Ker}\bar{\partial}_{1}\cap L^{n,q}_{(2)}(X,L)_{\psi}=\textrm{Ker}\bar{\partial}_{2}.
\end{equation}

Now the cohomology group can be expressed as
\[
H^{n,q}(X,L\otimes\mathscr{I}(\varphi))\simeq\frac{\textrm{Ker}\bar{\partial}_{1}}{\textrm{Im}\bar{\partial}_{1}}= \mathcal{H}^{n,q}_{\leqslant0}(X,L\otimes\mathscr{I}(L),\Delta_{0}))
\]
and
\[
H^{n,q}(X,L\otimes\mathfrak{I}(\varphi))\simeq\frac{\textrm{Ker}\bar{\partial}_{2}}{\textrm{Im}\bar{\partial}_{2}}= \textrm{Ker}\bar{\partial}_{2}\cap\textrm{Ker}\bar{\partial}^{\ast}_{\psi}.
\]
Therefore the morphism $i_{n,q}$ can be rewritten as
\[
i_{n,q}:\textrm{Ker}\bar{\partial}_{2}\cap\textrm{Ker}\bar{\partial}^{\ast}_{\psi}\rightarrow\mathcal{H}^{n,q}_{\leqslant0}(X,L\otimes \mathscr{I}(L),\Delta_{0})),
\]
hence its image equals
\[
\begin{split}
&\frac{\textrm{Ker}\bar{\partial}_{2}\cap\textrm{Ker}\bar{\partial}^{\ast}_{\psi}}{\textrm{Im}\bar{\partial}_{1}}\\
=&\frac{\textrm{Ker}\bar{\partial}_{2}/\textrm{Im}\bar{\partial}_{2}}{\textrm{Im}\bar{\partial}_{1}}=\frac{\textrm{Ker}\bar{\partial}_{2}} {\textrm{Im}\bar{\partial}_{1}}\\
=&\frac{\textrm{Ker}\bar{\partial}_{1}\cap L^{n,q}_{(2)}(X,L)_{\psi}}{\textrm{Im}\bar{\partial}_{1}}\\
=&\mathcal{H}^{n,q}_{\leqslant0}(X,L\otimes\mathscr{I}(L),\Delta_{0}))\cap L^{n,q}_{(2)}(X,L)_{\psi}.
\end{split}
\]
We use the fact that
\[
\textrm{Im}\bar{\partial}_{2}\subset\textrm{Im}\bar{\partial}_{1}
\]
to get the second equality. The third equality comes from formula (\ref{e24}) and the last equality is due to Lemma \ref{l41}. The proof is finished.
\end{proof}
\begin{problem}\label{q21}
We are willing to know whether we have
\[
\mathcal{H}^{n,q}_{\leqslant0}(X,L\otimes\mathfrak{I}(L))\simeq H^{n,q}(X,L\otimes\mathfrak{I}(L)).
\]
\end{problem}

We are ready to prove Theorem \ref{t2} now.
\begin{proof}[Proof of Theorem \ref{t2}]
Apply Theorem \ref{t1} with $\lambda=0$, we then have
\[
h^{n,q}_{\leqslant0}(L^{k}\otimes E\otimes\mathfrak{I}(L^{k}))\leqslant Ck^{n-q}.
\]
Then the conclusion follows from Proposition \ref{p41} after we substitute $E\otimes K^{-1}_{X}$ for $E$.
\end{proof}

We shall list some applications of Theorem \ref{t2}. The first application will be on the extension problem of the holomorphic sections. In fact, we can prove a more general version of the Grauert--Riemenschneider conjecture.
\begin{proof}[Proof of Theorem \ref{t3}]
The first case is simple, so we omit it here.

In the second situation, we have
\[
h^{0,q}(L^{k}\otimes\mathfrak{I}(L^{k}))\leqslant Ck^{n-q}
\]
for $q>1$ when $i_{0,q}$ is injective. We use the Riemann--Roch formula involving the ideal sheaf. Then we have
\[
\begin{split}
   h^{0}(L^{k}\otimes\mathfrak{I}(L^{k}))&=\chi(X,L^{k}\otimes\mathfrak{I}(L^{k}))+h^{1}(L^{k}\otimes\mathfrak{I}(L^{k}))+O(k^{n-2})\\
               &=\chi(X,L^{k})-\chi(V,L^{k})+h^{1}(L^{k}\otimes\mathfrak{I}(L^{k}))+O(k^{n-2})\\
               &=\frac{k^{n}(L)^{n}}{n!}-\frac{k^{l}(L)^{l}}{l!}+h^{1}(L^{k}\otimes\mathfrak{I}(L^{k}))+O(k^{n-2}).
\end{split}
\]
Here $V=V(\mathfrak{I}(L^{k}))$ and $l=\dim V$. Notice that
\[
H^{0}(X,L^{k}\otimes\mathfrak{I}(L^{k}))\subset H^{0}(X,L^{k}),
\]
$L$ is big if and only if $(L)^{n}>0$. The surjectivity case is similar.

The same argument applies in the third situation, and the proof is complete.
\end{proof}

We prove a vanishing theorem to finish this section.

\begin{proof}[Proof of Theorem \ref{t4}]
Firstly, we claim that if $\mathcal{H}^{n,q}_{\leqslant0}(X,L\otimes\mathfrak{I}(L))$ is non-zero,
\[
h^{0}(X,L^{k-1}\otimes\mathfrak{I}(L^{k-1}))\leqslant\dim\mathcal{H}^{n,q}_{\leqslant0}(X,L^{k}\otimes\mathfrak{I}(L^{k})).
\]
In fact, let $\{s_{j}\}$ be a basis of $H^{0}(X,L^{k-1}\otimes\mathfrak{I}(L^{k-1}))$. Then for any $\alpha\in \mathcal{H}^{n,q}_{\leqslant0}(X,L\otimes\mathfrak{I}(L))$, $\{s_{j}\alpha\}$ is linearly independent in $\mathcal{H}^{n,q}_{\leqslant0}(X,L^{k}\otimes\mathfrak{I}(L^{k}))$ by Proposition \ref{p25}. It leads to the inequality.

Now suppose that $\mathcal{H}^{n,q}_{\leqslant0}(X,L\otimes\mathfrak{I}(L))$ is non-zero for $q>n-\kappa(L)$. We have
\[
  h^{0}(X,L^{k-1})=h^{0}(X,L^{k-1}\otimes\mathfrak{I}(L^{k-1}))\leqslant\dim\mathcal{H}^{n,q}_{\leqslant0}(X,L^{k}\otimes\mathfrak{I}(L^{k})).
\]
The first equality comes from the assumption that $h_{0}$ provides an analytic Zariski decomposition, and the second inequality is due to the claim. By the definition of Iitaka dimension $\kappa(L)$, we have
\[
  \limsup_{k\rightarrow\infty}\frac{h^{0}(X,L^{k-1})}{(k-1)^{\kappa(L)}}>0.
\]
It means that
\[
  \limsup_{k\rightarrow\infty}\frac{\dim\mathcal{H}^{n,q}_{\leqslant0}(X,L^{k}\otimes\mathfrak{I}(L^{k}))}{(k-1)^{\kappa(L)}}>0.
\]
On the other hand, we have
\[
  \dim\mathcal{H}^{n,q}_{\leqslant0}(X,L^{k}\otimes\mathfrak{I}(L^{k}))\leqslant Ck^{n-q}
\]
by Theorem \ref{t2}, so $n-q\geqslant\kappa(L)$. It contradicts to the assumption that $q>n-\kappa(L)$. Hence
\[
\mathcal{H}^{n,q}_{\leqslant0}(X,L\otimes\mathfrak{I}(L))=\mathrm{Im}\mathrm{i}_{n,q}=0
\]
for $q>n-\kappa(L)$. Consider the cohomology long exact sequence of the short exact sequence
\[
0\rightarrow\mathfrak{I}(L)\rightarrow\mathscr{I}(L)\rightarrow\mathfrak{I}(L)/\mathscr{I}(L)\rightarrow0.
\]
Notice that $\textrm{supp}(\mathfrak{I}(L)/\mathscr{I}(L))=\{h_{0}=\infty\}$, the conclusion follows immediately from the fact that $\mathrm{Im}\mathrm{i}_{n,q}=0$ and $H^{q}(X,\mathfrak{I}(L)/\mathscr{I}(L))=0$ when $q>m$.
\end{proof}

\section{The pseudo-effective case}
In this section, we will discuss the situation that $L$ is merely pseudo-effective.
\subsection{The harmonic forms}
As we have shown before, the ingredient to define a Laplacian operator as well as the associated eigenform for a singular metric $\phi$ is to approximate it by a family of smooth metrics $\{\phi_{\varepsilon}\}$. The difference for a pseudo-effective line bundle is that we can do such an approximation, only on an open subvariety $Y\subset X$. However, it seems to be enough, at least to define the harmonic $L$-valued $(n,q)$-forms.

Now let $(L,\phi)$ be a pseudo-effective line bundle on a compact complex manifold $X$. Assume that there exits a holomorphic section $s$ of $L^{k_{0}}$ for some integer $k_{0}$, such that $\sup_{X}|s|_{k_{0}\phi}<\infty$. Fix a Hermitian metric $\omega$ on $X$. Then by Demailly's approximation \cite{DPS01}, we can find a family of metrics $\{\phi_{\varepsilon}\}$ on $L$ with the following properties:

(a) $\phi_{\varepsilon}$ is smooth on $X-Z_{\varepsilon}$ for a subvariety $Z_{\varepsilon}$;

(b) $\phi\leqslant\phi_{\varepsilon_{1}}\leqslant\phi_{\varepsilon_{2}}$ holds for any $0<\varepsilon_{1}\leqslant\varepsilon_{2}$;

(c) $\mathscr{I}(\phi)=\mathscr{I}(\phi_{\varepsilon})$; and

(d) $i\Theta_{L,\phi_{\varepsilon}}\geqslant-\varepsilon\omega$.

Thanks to the proof of the openness conjecture by Berndtsson \cite{Ber15}, one can arrange $\phi_{\varepsilon}$ with logarithmic poles along $Z_{\varepsilon}$ according to the remark in \cite{DPS01}. Moreover, since the norm $|s|_{k_{0}\phi}$ is bounded on $X$, the set $\{x\in X|\nu(\phi_{\varepsilon},x)>0\}$ for every $\varepsilon>0$ is contained in the subvariety $Z:=\{x|s(x)=0\}$ by property (b). Here $\nu(\phi_{\varepsilon},x)$ refers to the Lelong number of $\phi_{\varepsilon}$ at $x$. Hence, instead of (a), we can assume that

(a') $\phi_{\varepsilon}$ is smooth on $X-Z$, where $Z$ is a subvariety of $X$ independent of $\varepsilon$.

Now let $Y=X-Z$. We use the method in \cite{Dem82} to construct a complete Hermitian metric on $Y$ as follows. Since $Y$ is weakly pseudo-convex, we can take a smooth plurisubharmonic exhaustion function $\psi$ on $X$. Define $\tilde{\omega}=\omega+\frac{1}{l}i\partial\bar{\partial}\psi^{2}$ for $l\gg0$. It is easy to verify that $\tilde{\omega}$ is a complete Hermitian metric on $Y$ and $\tilde{\omega}\geqslant\omega$.

Let $L^{n,q}_{(2)}(Y,L)_{\phi_{\varepsilon},\tilde{\omega}}$ be the $L^{2}$-space of $L$-valued $(n,q)$-forms $u$ on $Y$ with respect to the inner product given by $\phi_{\varepsilon},\tilde{\omega}$. Then we have the orthogonal decomposition
\begin{equation}\label{e13}
L^{n,q}_{(2)}(Y,L)_{\phi_{\varepsilon},\tilde{\omega}}=\mathrm{Im}\bar{\partial}\bigoplus\mathcal{H}^{n,q}_{\phi_{\varepsilon}, \tilde{\omega}}(L)\bigoplus\mathrm{Im}\bar{\partial}^{\ast}_{\phi_{\varepsilon}}
\end{equation}
where
\[
  \mathcal{H}^{n,q}_{\phi_{\varepsilon}, \tilde{\omega}}(L)=\{\alpha|\bar{\partial}\alpha=0, \bar{\partial}^{\ast}_{\phi_{\varepsilon}}\alpha=0\}.
\]
We give a brief explanation for decomposition (\ref{e13}). Usually $\mathrm{Im}\bar{\partial}$ is not closed in the $L^{2}$-space of a noncompact manifold even if the metric is complete. However, in the situation we consider here, $Y$ has the compactification $X$, and the forms on $Y$ are bounded in $L^{2}$-norms. Such a form will have good extension properties. Therefore the set $L^{n,q}_{(2)}(Y,L)_{\phi_{\varepsilon},\tilde{\omega}}\cap\mathrm{Im}\bar{\partial}$ behaves much like the space $\mathrm{Im}\bar{\partial}$ on $X$, which is surely closed. The complete explanation can be found in \cite{Fuj12,Wu17}.

Now we have all the ingredients for the definition of $\Delta_{\phi}$-harmonic forms. We denote the Lapalcian operator on $Y$ associated to $\tilde{\omega}$ and $\phi_{\varepsilon}$ by $\Delta_{\varepsilon}$.
\begin{definition}\label{d51}
Let $\alpha$ be an $L$-valued $(n,q)$-form on $X$ with bounded $L^{2}$-norm with respect to $\omega,\phi$. Assume that for every $\varepsilon\ll1$, there exists a Dolbeault cohomological equivalent class $\alpha_{\varepsilon}\in[\alpha|_{Y}]$ such that
\begin{enumerate}
  \item $\Delta_{\varepsilon}\alpha_{\varepsilon}=0$ on $Y$;
  \item $\alpha_{\varepsilon}\rightarrow\alpha|_{Y}$ in $L^{2}$-norm.
\end{enumerate}
Then we call $\alpha$ a $\Delta_{\phi}$-harmonic form. The space of all the $\Delta_{\phi}$-harmonic forms is denoted by
\[
\mathcal{H}^{n,q}_{\leqslant0}(X,L\otimes\mathscr{I}(\phi),\Delta_{\phi}).
\]
\end{definition}

We will show that Definition \ref{d51} is compatible with the usual definition of $\Delta_{\phi}$-harmonic forms for a smooth $\phi$ by proving the following Hodge-type isomorphism. Notice that here we furthermore assume that $(X,\omega)$ is K\"{a}hler.
\begin{proposition}[A singular version of Hodge's theorem, III]\label{p51}
Let $(X,\omega)$ be a compact K\"{a}hler manifold. $(L,\phi)$ is a pseudo-effective line bundle on $X$. Assume that there exists a section $s$ of some multiple $L^{k}$ such that $\sup_{X}|s|_{k\phi}<\infty$. Then the following isomorphism holds:
\begin{equation}\label{e14}
   \mathcal{H}^{n,q}_{\leqslant0}(X,L\otimes\mathscr{I}(\phi),\Delta_{\phi})\simeq H^{n,q}(X,L\otimes\mathscr{I}(\phi)).
\end{equation}
In particular, when $\phi$ is smooth, $\alpha\in\mathcal{H}^{n,q}_{\leqslant0}(X,L,\Delta_{\phi})$ if and only if $\alpha$ is $\Delta_{\phi}$-harmonic in the usual sense.
\begin{proof}
We use the de Rham--Weil isomorphism
\[
H^{n,q}(X,L\otimes\mathscr{I}(\phi))\cong\frac{\mathrm{Ker}\bar\partial\cap L^{n,q}_{(2)}(X,L)_{h,\omega}}{\mathrm{Im}\bar{\partial}}
\]
to represent a given cohomology class $[\alpha]\in H^{n,q}(X,L\otimes\mathscr{I}(\phi))$ by a $\bar{\partial}$-closed $L$-valued $(n,q)$-form $\alpha$ with $\|\alpha\|_{\phi,\omega}<\infty$. We denote $\alpha|_{Y}$ simply by $\alpha_{Y}$. Since $\tilde{\omega}\geqslant\omega$, it is easy to verify that
\[
|\alpha_{Y}|^{2}_{\phi_{\varepsilon},\tilde{\omega}}dV_{\tilde{\omega}}\leqslant|\alpha|^{2}_{\phi_{\varepsilon},\omega}dV_{\omega},
\]
which leads to inequality $\|\alpha_{Y}\|_{\phi_{\varepsilon},\tilde{\omega}}\leqslant\|\alpha\|_{\phi_{\varepsilon,\omega}}$ with $L^{2}$-norms. Hence by property (b), we have
$\|\alpha_{Y}\|_{\phi_{\varepsilon},\tilde{\omega}}\leqslant\|\alpha\|_{\phi,\omega}$ which implies
\[
\alpha_{Y}\in L^{n,q}_{(2)}(Y,L)_{\phi_{\varepsilon},\tilde{\omega}}.
\]
By decomposition (\ref{e13}), we have a harmonic representative $\alpha_{\varepsilon}$ in
\[
\mathcal{H}^{n,q}_{\phi_{\varepsilon,\tilde{\omega}}}(L),
\]
which means that $\Delta_{\varepsilon}\alpha_{\varepsilon}=0$ on $Y$ for all $\varepsilon$. Moreover, since a harmonic representative minimizes the $L^{2}$-norm, we have
\[
   \|\alpha_{\varepsilon}\|_{\phi_{\varepsilon},\tilde{\omega}}\leqslant\|\alpha_{Y}\|_{\phi_{\varepsilon},\tilde{\omega}}\leqslant \|\alpha\|_{\phi,\omega}.
\]
So we can take the limit $\tilde{\alpha}$ of (a subsequence of) $\{\alpha_{\varepsilon}\}$ such that
\[
\tilde{\alpha}\in[\alpha_{Y}].
\]
It is left to extend it to $X$.

Indeed, by (the proof of) Proposition 2.1 in \cite{Wu17}, there is an injective morphism, which maps $\tilde{\alpha}$ to a $\bar{\partial}$-closed $L$-valued $(n-q,0)$-form on $Y$ with bounded $L^{2}$-norm. We formally denote it by $\ast\tilde{\alpha}$. The canonical extension theorem applies here and $\ast\tilde{\alpha}$ extends to a $\bar{\partial}$-closed $L$-valued $(n-q,0)$-form on $X$, which is denoted by $S^{q}(\tilde{\alpha})$ in \cite{Wu17}. Furthermore, it is shown by Proposition 2.2 in \cite{Wu17} that $\hat{\alpha}:=c_{n-q}\omega_{q}\wedge S^{q}(\tilde{\alpha})$ is an $L$-valued $(n,q)$-form with
\[
  \hat{\alpha}|_{Y}=\tilde{\alpha}.
\]
Therefore we finally get an extension $\hat{\alpha}$ of $\tilde{\alpha}$. By definition,
\[
\hat{\alpha}\in\mathcal{H}^{n,q}_{\leqslant0}(X,L\otimes\mathscr{I}(\phi),\Delta_{\phi}).
\]
We denote this morphism by $i([\alpha])=\hat{\alpha}$.

On the other hand, for a given $\alpha\in\mathcal{H}^{n,q}_{\leqslant0}(X,L\otimes\mathscr{I}(\phi),\Delta_{\phi})$, by definition there exists an $\alpha_{\varepsilon}\in[\alpha_{Y}]$ with $\alpha_{\varepsilon}\in\mathcal{H}^{n,q}_{\phi_{\varepsilon,\tilde{\omega}}}(L)$ for every $\varepsilon$. In particular, $\bar{\partial}\alpha_{\varepsilon}=0$. So all of the $\alpha_{\varepsilon}$ together with $\alpha_{Y}$ define a common cohomology class $[\alpha_{Y}]$ in $H^{n,q}(Y,L\otimes\mathscr{I}(\phi))$. It is left to extend this class to $X$.

We use the $S^{q}$ again. It maps $[\alpha_{Y}]$ to
\[
S^{q}(\alpha_{Y})\in H^{0}(X,\Omega^{n-q}_{X}\otimes L\otimes\mathscr{I}(\phi)).
\]
Furthermore,
\[
   c_{n-q}\omega_{q}\wedge S^{q}(\alpha_{Y})\in H^{n,q}(X,L\otimes\mathscr{I}(\phi))
\]
with $[(c_{n-q}\omega_{q}\wedge S^{q}(\alpha_{Y}))|_{Y}]=[\alpha_{Y}]$. Here we use the fact that $\omega$ is a K\"{a}hler metric. We denote this morphism by $j(\alpha)=[c_{n-q}\omega_{q}\wedge S^{q}(\alpha_{Y})]$. It is easy to verify that $i\circ j=\textrm{id}$ and $j\circ i=\textrm{id}$. The proof is finished.
\end{proof}
\end{proposition}
\begin{remark}\label{r51}
When $\phi$ has analytic singularities, we can use the same method as in Proposition \ref{p41} to extend this isomorphism to $\mathfrak{I}(\phi)$.
\end{remark}
\begin{remark}\label{r52}
In \cite{Mat14}, a similar result (Lemma 3.2) has been shown for a line bundle $(L,h)$ such that $h$ is smooth outside a subvariety and $i\Theta_{L,h}\geqslant0$. Our result, which benefits a lot form their work, generalizes it.
\end{remark}
Although Proposition \ref{p51} only holds for the $(n,q)$-form, we remark that the estimate for $h^{n,q}_{\leqslant0}(L^{k}\otimes E)$ is enough to get the estimate for all of the $\dim H^{p,q}(L^{k}\otimes E)$ since we can substitute $E$ by $E\otimes\Omega^{p}_{X}\otimes K^{-1}_{X}$. So it remains to prove an estimate for $h^{n,q}_{\leqslant0}(L^{k}\otimes E)$.
\subsection{The estimate for the order of the cohomology group}
Remember that our method highly depends on a submeanvalue inequality (Proposition \ref{p32}). However, there is a gap when proving such an inequality on an open subset.

In fact, in order to prove Proposition \ref{p32}, we first use the $\partial\bar{\partial}$-Bochner formula to reduce it to a ball with radius $r$ (Proposition \ref{p31}). If we want to get a similar inequality on $Y$, first we need to equip $Y$ with a complete Hermitian metric $\tilde{\omega}$. At this time, it is sort of like to estimate
\begin{equation}\label{e15}
   \int_{|z|<1}|F||\nabla h_{\varepsilon}|^{2}\tilde{\omega}_{n}.
\end{equation}
Here $F$ is the multiplier and describes the local rescalings of infinitesimally small coordinate charts. When the first derivative $\nabla h_{\varepsilon}$ becomes large as the point approaching $Z$ and $\varepsilon$ tending zero, to make the $L^{2}$-norm bounded, we have to enlarge the coordinate in that direction at that point. It is the same as collapsing the manifold along that direction at that point. When we fix our sight on the manifold, $\nabla h_{\varepsilon}$ blows up, but when we fix our sight on $\nabla h_{\varepsilon}$, the manifold collapses. As a result, the integral (\ref{e15}) is hard to control. So it is still an open question to get an estimate for $h^{n,q}_{\leqslant\lambda}(L^{k}\otimes E)$.

\address{

\small Current address: School of Mathematical Sciences, Fudan University, Shanghai 200433, People's Republic of China.

\small E-mail address: jingcaowu08@gmail.com, jingcaowu13@fudan.edu.cn
}

\end{document}